\documentclass[a4paper,10pt]{article}

\usepackage{amsthm}
\usepackage{thmtools}
\usepackage{mathrsfs}
\usepackage{url}
\usepackage{amsmath}
\usepackage{amsxtra}
\usepackage{amssymb}
\usepackage{turnstile}
\usepackage{enumitem}
\usepackage[retainorgcmds]{IEEEtrantools}
\usepackage{hyperref}
\usepackage{accents}
\declaretheoremstyle[bodyfont=\sl]{slanted}

\swapnumbers
\declaretheorem[name=Definition,style=definition,qed=$\dashv$,
numberwithin=section]{dfn}
\declaretheorem[name=Definition,style=definition,numbered=no,qed=$\dashv$]{dfn*}
\declaretheorem[name=Definition,style=definition,numbered=no]{dfnnoqed*}

\declaretheorem[name=Theorem,style=slanted,sibling=dfn]{tm}
\declaretheorem[name=Theorem,style=slanted,numbered=no]{tm*}

\declaretheorem[name=Corollary,style=slanted,numbered=no]{cor*}
\declaretheorem[name=Remark,style=definition,sibling=dfn]{rem}
\declaretheorem[name=Question,style=definition,sibling=dfn]{ques}

\swapnumbers
\declaretheoremstyle[headfont=\scshape]{claimstyle}
\declaretheorem[name=Claim,style=claimstyle]{clm}

\declaretheorem[name=Claim,style=claimstyle]{clmtwo}
\declaretheorem[name=Claim,style=claimstyle]{clmthree}

\declaretheorem[name=Claim,style=claimstyle]{clmfive}

\declaretheorem[name=Claim,style=claimstyle,numbered=no]{clm*}
\declaretheorem[name=Subclaim,style=claimstyle,numberwithin=clm]{sclm}

\declaretheorem[name=Subclaim,style=claimstyle,numberwithin=clmtwo]{sclmtwo}

\declaretheorem[name=Subclaim,style=claimstyle,numbered=no]{sclm*}

\declaretheorem[name=Subsubclaim,style=claimstyle,numbered=no]{ssclm*}

\declaretheoremstyle[headfont=\scshape]{casestyle}

\newcommand{\measlim}{\mathrm{ml}}
\newcommand{\successor}{\mathrm{succ}}

\newcommand{\GCH}{\mathrm{GCH}}

\newcommand{\swsw}{\mathrm{swsw}}

\newcommand{\Mmm}{\mathscr{M}}

\newcommand{\CC}{\mathbb C}
\newcommand{\QQ}{\mathbb Q}
\newcommand{\RR}{\mathbb R}

\newcommand{\PP}{\mathbb P}
\newcommand{\BB}{\mathbb B}
\newcommand{\sub}{\subseteq}
\newcommand{\cross}{\times}

\newcommand{\compat}{\parallel}
\newcommand{\incompat}{\perp}
\newcommand{\inter}{\cap}
\renewcommand{\int}{\inter}

\newcommand{\om}{\omega}
\newcommand{\pow}{\mathcal{P}}
\newcommand{\OR}{\mathrm{OR}}

\newcommand{\Hull}{\mathrm{Hull}}

\newcommand{\cut}{\backslash}

\newcommand{\Tt}{\mathcal{T}}
\newcommand{\Ss}{\mathcal{S}}
\newcommand{\Uu}{\mathcal{U}}

\newcommand{\rg}{\mathrm{rg}}

\newcommand{\pins}{\triangleleft}

\newcommand{\crit}{\mathrm{cr}}

\newcommand{\rest}{\!\upharpoonright\!}
\newcommand{\com}{\circ}

\newcommand{\lh}{\mathrm{lh}}
\newcommand{\Ult}{\mathrm{Ult}}

\newcommand{\sats}{\models}
\newcommand{\elem}{\preccurlyeq}
\newcommand{\J}{\mathcal{J}}

\newcommand{\AC}{\mathrm{AC}}

\newcommand{\HOD}{\mathrm{HOD}}
\newcommand{\HC}{\mathrm{HC}}
\newcommand{\ZFC}{\mathrm{ZFC}}
\newcommand{\ZF}{\mathrm{ZF}}

\newcommand{\Coll}{\mathrm{Col}}
\newcommand{\es}{\mathbb{E}}

\newcommand{\core}{\mathfrak{C}}

\newcommand{\dirlim}{\mathrm{dir lim}}

\newcommand{\id}{\mathrm{id}}

\newcommand{\conc}{\ \widehat{\ }\ }

\newcommand{\forces}{\dststile{}{}}

\DeclareMathOperator{\card}{card}
\DeclareMathOperator{\cof}{cof}
\DeclareMathOperator{\wfp}{wfp}

\newcommand{\OD}{\mathrm{OD}}

\newcommand{\psub}{\subsetneq}

\newcommand{\cHull}{\mathrm{cHull}}

\newcommand{\Fff}{\mathscr{F}}
\newcommand{\Ddd}{\mathscr{D}}

\newcommand{\concatB}{
  \mathbin{\rotatebox[origin=c]{90}{\scalebox{.7}{(\kern1ex)}}}
}

\title{Local mantles of $L[x]$}

\author{Farmer Schlutzenberg\footnote{
Gef\"ordert durch die Deutsche Forschungsgemeinschaft (DFG) -- Projektnummer 445387776.
Funded by the Deutsche Forschungsgemeinschaft (DFG, German Research Foundation) -- project number 445387776. This work supported by  Deutsche Forschungsgemeinschaft (DFG, German
Research
Foundation) under Germany's Excellence Strategy EXC 2044-390685587,
Mathematics M\"unster: Dynamics-Geometry-Structure.
Editing work funded by the Austrian Science Fund (FWF) [10.55776/Y1498].}\\
farmer.schlutzenberg@gmail.com}

\begin{document}

\maketitle

\begin{abstract}Assume $\ZFC$. Let $\kappa$ be a cardinal. Recall that a
\emph{${<\kappa}$-ground}
is  a transitive proper class $W$ modelling $\ZFC$ such that $V$ is a generic
extension of $W$ via a forcing $\PP\in W$ of cardinality
${<\kappa}$, and
the \emph{$\kappa$-mantle} $\Mmm_\kappa$ is the intersection
of all ${<\kappa}$-grounds.

Assume there is a Woodin cardinal and a proper class of measurables,
and let $x$ be a real of sufficiently high Turing degree.
Let $\kappa$ be a limit cardinal of $L[x]$
of uncountable cofinality in $L[x]$.
Using methods from Woodin's 
analysis of $\HOD^{L[x,G]}$, we analyze
 $\Mmm_\kappa^{L[x]}$, and show that it models ZFC + GCH + ``There is a Woodin cardinal''.
Moreover, we  show that it is a fully iterable strategy mouse,
(analogous to $\HOD^{L[x,G]}$).
We also analyze another form of ``local mantle'', partly assuming a weak form of Turing determinacy.

We also compute bounds on how much iteration strategy
can be added to $M_1$ before  $M_1^\#$ is added.\end{abstract}

\section{Introduction}

We assume the reader is familiar with the basic notions
of set-theoretic geology and of inner model theory,
and with Woodin's analysis of $\HOD^{L[x,G]}$.
For background on these, see \cite{stgeol}, \cite{outline}, \cite{hod_as_core_model}.
 Regarding terminology relating to grounds, 
mantles, etc, we follow \cite{choice_principles_in_local_mantles}.

Usuba \cite{usuba_ddg}, \cite{usuba_extendible} showed that if $\kappa$ is a strong limit cardinal
then the $\kappa$-mantle $\Mmm_\kappa$ (the intersection of all ${<\kappa}$-grounds) models $\ZF$, and the full mantle $\Mmm=\Mmm_\infty\sats\ZFC$.
A basic question, then, is whether
$\Mmm_\kappa\sats\AC$. Usuba showed that
if $\kappa$ is extendible, then $\Mmm_\kappa=\Mmm$, and hence $\Mmm_\kappa\sats\AC$.
Lietz \cite{lietz} showed that it is consistent relative to a Mahlo cardinal
that $\kappa$ is Mahlo but $\Mmm_\kappa\sats\neg\AC$.
The author, using the general theory of \cite{vm2_v2}, then
analysed $\mathscr{M}_{\kappa_0}^M$ for a certain mouse $M$\footnote{The mouse
	$M_{\swsw}$, the least proper class mouse satisfying ``there are $\delta_0<\kappa_0<\delta_1<\kappa_1$ such that each $\delta_i$ is Woodin
	and each $\kappa_i$ is strong''.} and the least strong cardinal $\kappa_0$ of $M$, showing in particular
that it models $\ZFC$ (this will appear in \cite{vm2_v2}).
Schindler \cite{schindler_fsttimtg} showed that if $\kappa$ is measurable then $\Mmm_\kappa\sats\AC$.
The author \cite{choice_principles_in_local_mantles} then established that some weaker variants of $\AC$
follow from the weak compactness and inaccessibility of $\kappa$
respectively.

Also, in \cite{vm0}, Schindler and Fuchs analyzed the full mantle
$\mathscr{M}^{L[x]}$ of $L[x]$, for a cone of reals $x$,
under large cardinal assumptions. They show that it is a fine structural
model, which models GCH, but does not have large cardinals.

In the present paper we consider  ``local mantles'' of $L[x]$
such as $\Mmm_\kappa^{L[x]}$,
where $x$ is a real of high complexity, under large cardinal hypotheses.
We will only
discuss $\Mmm_\kappa^{L[x]}$ itself when $\kappa$ is a limit cardinal (we do not
know how to handle successor cardinals),
and we will mostly focus on the case that $\kappa$ has uncountable
cofinality in $L[x]$.
But we will also deal with a more refined ``local mantle''
instead of $\Mmm_\kappa^{L[x]}$:

\begin{dfn}
Assume $\ZFC$. Given some class $\mathscr{G}$ of grounds,
the \emph{$\mathscr{G}$-mantle} 
$\Mmm_\mathscr{G}$ is $\bigcap\mathscr{G}$.

Let $\delta>\om$ be regular.
An \emph{internally-$\delta$-cc ground}
 is a ground $W$,
 as witnessed by a forcing
  $\PP\in W$
 such that $W\sats$``$\PP$ is $\delta$-cc''.
We write $\mathscr{G}^{\mathrm{int}}_{\delta\text{-cc}}$
for the class of these grounds, and 
$\Mmm^{\mathrm{int}}_{\delta\text{-cc}}=\Mmm_{\mathscr{G}^{\mathrm{int}}_{\delta\text{-cc}}}$.

An \emph{externally-$\delta$-cc ground}
is likewise, but we replace the demand
that $W\sats$``$\PP$ is $\delta$-cc''
with $V\sats$``$\PP$ is $\delta$-cc''.
We write
$\mathscr{G}^{\mathrm{ext}}_{\delta\text{-cc}}$
for the class of these grounds, and 
$\Mmm^{\mathrm{ext}}_{\delta\text{-cc}}=\Mmm_{\mathscr{G}^{\mathrm{ext}}_{\delta\text{-cc}}}$.
\end{dfn}

Note that   
$\mathscr{G}^{\text{ext}}_{\delta\text{-cc}}\sub\mathscr{G}^{\text{int}}
_{\delta\text{-cc}}$,
so $\Mmm^{\text{ext}}_{\delta\text{-cc}}\supseteq\Mmm^{\text{int}}_{\delta\text{-cc}}$.

The main theorems and arguments in the paper involve
direct limit systems of mice, and the analysis of the local mantles
relate naturally to that of the full mantles of $L[x]$ described by Schindler and Fuchs. But we state a coarse version of the theorems
first, that makes no mention of mice. \emph{Constructible Turing determinacy}
is just a weak form of Turing determinacy, for which see Definition \ref{dfn:con_Turing}.
\begin{tm}\label{tm:coarse}
	Assume there is a Woodin cardinal and a measurable cardinal above it.
	
	Then for a Turing cone of reals $x$, for every limit cardinal
	$\eta$ of $L[x]$ such that $L[x]\sats$``$\eta$ has uncountable cofinality'',
	letting $\delta=(\eta^+)^{L[x]}$, we have
	\[ \mathscr{M}_\eta^{L[x]}\sats\ZFC+\GCH+\text{``}\eta\text{ is the least measurable and }\delta\text{ the unique Woodin'',}\]
	and $\mathscr{M}_\eta^{L[x]}$ is an internally $\delta$-cc ground of $L[x]$, via a forcing $\sub\delta$.
	
	Assume also
 constructible Turing determinacy.
	Then for a Turing cone of reals $x$, letting $\theta_0$ be the least
	Mahlo cardinal of $L[x]$, for every $\eta\in(\om,\theta_0]$
	such that $\eta$ is regular in $L[x]$, letting $\delta=(\eta^{++})^{L[x]}$,
	we have:
	\begin{enumerate}[label=--]
		\item  $(\Mmm_{\eta\text{-cc}}^{\text{int}})^{L[x]}=(\Mmm_{\eta\text{-cc}}^{\text{ext}})^{L[x]}\sats\ZFC+\GCH+$``$\eta$ is the least
		measurable  and $\delta$ the unique Woodin'', and this model is an internally $\delta$-cc ground of $L[x]$, via a forcing $\sub\delta$,
		\item  for every internally $\eta$-cc ground $W$ 
		of $L[x]$, there are $A,B,\PP,\QQ\sub\eta$ such that:
		\begin{enumerate}[label=--]
			\item $L[A]\sub L[B]=W$,
			\item $L[A]$ is an externally $\eta$-cc ground of $L[x]$, as witnessed by $\PP$,
			and\item $W$ is an internally $\eta$-cc ground of $L[x]$, as witnessed by $\QQ$.
		\end{enumerate}
\end{enumerate}\end{tm}

Note that in the first part above, $\delta=(\eta^+)^{L[x]}$,
while in the second, $\delta=(\eta^{++})^{L[x]}$.

Now recall that $M_1$ is the minimal proper class mouse
with a Woodin cardinal, and $M_1^\#$  its sharp,
which is the least active mouse with a Woodin cardinal.
We generally assume that $M_1^\#$ exists and is fully iterable;
that is, $(0,\OR,\OR)$-iterable. One can also prove more local
versions of the theorems which assume only set-much iterability, but
we leave this to the reader. (In fact for Theorem \ref{tm:coarse},
we did not quite have this assumption, as it implies that every set
has a sharp. But for that theorem, the lesser assumptions suffice.)
For any model $N$ with a unique Woodin cardinal,
we write $\delta^N$ for that Woodin.
Let $x\in\RR$ with $M_1\sub L[x]$.
Given an $L[x]$-cardinal $\delta$ with $\om<\delta^{M_1}\leq\delta$,
let $\mathscr{F}^{L[x]}_{\leq\delta}$
be the ``set'' of all non-dropping iterates
$N$ of $M_1$ such that $N|\delta^N\in L[x]$
and $\delta^N\leq\delta$. Let $\mathscr{F}^{L[x]}_{<\delta}$
be likewise, but demanding $\delta^N<\delta$.
Given $\mathscr{F}=\mathscr{F}^{L[x]}_{\leq\delta}$
or $\mathscr{F}=\mathscr{F}^{L[x]}_{<\delta}$,
assuming that $\mathscr{F}$ is directed,
let $M_\infty^{\mathscr{F}}$ be the associated direct limit.

Again assuming that $\mathscr{F}$ is directed,
we will define in \S\ref{sec:background} an associated fragment
$\Sigma^{\mathscr{F}}_\infty$ of the iteration strategy
for $M_\infty^{\mathscr{F}}$. Most
of the following theorem is well-known, and (a trivial variant of
a result) due to Hugh Woodin,
through his analysis of $\HOD^{L[x,G]}$; the only 
new thing being added is the equation $\mathscr{M}^{L[x]}_\kappa=\HOD^{L[x,G]}$,
but we state the remaining facts as a reminder and for comparison
with the theorem that follows it:

\begin{tm}\label{tm:unctbl_cof}
 Assume $M_1^\#$ exists and is fully iterable and let $x\in\RR$
 with $M_1\sub L[x]$.
 Let $\kappa$ be
 such that $L[x]\sats$``$\kappa$ is a limit cardinal of uncountable 
cofinality'' and
 $\delta^{M_1}<\kappa$. Let $G$ be $(L[x],\Coll(\om,{<\kappa}))$-generic
 and $\RR^+$ be the corresponding symmetric reals; that is,
 \[ \RR^+=\bigcup_{\alpha<\kappa}\RR\inter L[x,G\rest\alpha].\]
Then
\begin{enumerate}[label=--]
 \item $\mathscr{F}=\mathscr{F}^{L[x]}_{<\kappa}$ is directed,
 \item 
$\Mmm_\kappa^{L[x]}=\HOD^{L(\RR^+)}=M_\infty^{\mathscr{F}}[\Sigma^{\mathscr{F}}_\infty]
\sats\ZFC$,\footnote{Suppose $\kappa$ is inaccessible
	in $L[x]$. Then  $\HOD^{L(\RR^+)}=\HOD^{L[x,G]}$.
	In the case of the least inaccessible, this is mentioned in
	\cite[p.~196]{lcfd}. Note that
 (by inaccessibility) $\RR^+=\RR^{L[x,G]}$, so immediately
	 $\HOD^{L(\RR^+)}\sub\HOD^{L[x,G]}$.
	 But it is an easy consequence of Woodin's analysis (and its generalization)
	 that $\HOD^{L[x,G]}\sub\HOD^{L(\RR^+)}$. Alternatively,
	  there is also a short direct argument: Working in $L(\RR^+)$,
	  argue via homogeneity that there is a real $z$ (in fact, $x$) such that for
	  all reals $y_0,y_1\geq_T z$ and all 
 generic filters $H_i$ for $(L[y_i],\Coll(\om,{<\kappa}))$,
we have  $\HOD^{L[y_0,H_0]}=\HOD^{L[y_1,H_1]}=\HOD^{L[x,G]}$.

However, if $\kappa$ is singular, then $\kappa$ has cofinality
$\omega$ in $L[x,G]$, so $\RR^+\psub\RR^{L[x,G]}$, and in this case
it is important that we refer to $\HOD^{L(\RR^+)}$ instead of $\HOD^{L[x,G]}$.} 
\item $\delta_\infty=(\kappa^+)^{L[x]}$ is the unique Woodin cardinal
of $M_\infty^{\mathscr{F}}[\Sigma^{\mathscr{F}}_\infty]$,
\item $M_\infty^{\mathscr{F}}[\Sigma^{\mathscr{F}}_\infty]$
is an internally $\delta_\infty$-cc ground of $L[x]$ via forcing $\sub\delta_\infty$,
\item $M_\infty^{\mathscr{F}}[\Sigma^{\mathscr{F}}_\infty]$ is
fully iterable.
\end{enumerate}
\end{tm}

\begin{dfn}\label{dfn:con_Turing}
	\emph{Constructible Turing determinacy}
	is the assertion that for every real $a$ and ordinal $\eta$
	and formula 
	$\varphi$
	in the language of set theory,
	either there is a cone of reals $x$ such that $L[a,x]\sats\varphi(a,\eta)$,
	or there is a cone of reals $x$ such that 
	$L[a,x]\sats\neg\varphi(a,\eta)$.
\end{dfn}

\begin{tm}\label{tm:delta-cc_mantle}
 Assume $M_1^\#$ exists and is fully iterable.
 Let $x\in\RR$ and $N$ be a non-dropping iterate
 of $M_1$ with $N\sub L[x]$
 and $\delta^{N}\leq\delta\leq\theta_0$, where $\theta_0$ is the least Mahlo cardinal of $L[x]$,
 and $\delta$ is regular in $L[x]$.
Then:
\begin{enumerate}[label=--]
 \item $\mathscr{F}=\mathscr{F}^{L[x]}_{\leq\delta}$ is directed,
\item 
$(\Mmm^{\text{int}}_\delta)^{L[x]}=(\Mmm^{\text{ext}}_\delta)^{L[x]}=
M_\infty^{\mathscr{F}}[\Sigma^{\mathscr{F}}_\infty]$,
\item  
$\delta_\infty$ is the unique Woodin cardinal of $M_\infty^{\mathscr{F}}[\Sigma^{\mathscr{F}}_\infty]$,
\item $M_\infty^{\mathscr{F}}[\Sigma^{\mathscr{F}}_\infty]$
is an internally $\delta_\infty$-cc ground of $L[x]$, via a forcing $\sub\delta_\infty$,
\item $M_\infty^{\mathscr{F}}[\Sigma^{\mathscr{F}}_\infty]$
is fully iterable,
\item either
$\delta_\infty=\delta^{+L[x]}$ or $\delta_\infty=\delta^{++L[x]}$,
and if constructible Turing determinacy holds then
$\delta_\infty=\delta^{++L[x]}$.
\end{enumerate}
\end{tm}

In particular, if $\delta=\om_1^{L[x]}$ in the theorem above,
then $\delta_\infty=\om_3^{L[x]}$. We do not know the identity of $\om_2^{L[x]}$ in $M_\infty^{\mathscr{F}}$,
but it is strictly bigger than its  least strong cardinal
(see \S\ref{sec:strong_cardinal}).

One can also consider strategy extensions $M_1[\Sigma]$ of $M_1$ itself.
Woodin showed that if one constructs from $M_1$ and completely
closes under the strategy $\Sigma_{M_1}$ for $M_1$,
one eventually constructs
$M_1^\#$ (in fact $L[M_1^\#]=M_1[\Sigma_\OR]$, where $\Sigma_\OR$ is specified below). 
In \S\ref{sec:strategy_extensions} we find reasonably tight bounds
on
how  much strategy is enough to reach $M_1^\#$.

\begin{dfn}\label{dfn:nice_extension}
Assume that $M_1^\#$ exists and is fully iterable.
We write $\Sigma_{M_1}$ for the unique normal (that is, $(0,\OR)$-) iteration strategy for $M_1$.
Given an $M_1$-cardinal $\kappa$ or $\kappa=\OR$, let 
$\Sigma_\kappa=\Sigma_{M_1}\rest(M_1|\kappa)$.
Say that $M_1[\Sigma_\kappa]$ is a \emph{nice extension}
iff $V_{\delta^{M_1}}^{M_1[\Sigma_\kappa]}=V_{\delta^{M_1}}^{M_1}$,
$\delta^{M_1}$ is Woodin in $M_1[\Sigma_\kappa]$,
and $\Sigma_{M_1}$ induces a (normal) iteration strategy $\Sigma^+$ for 
$M_1[\Sigma_\kappa]$.
\end{dfn}

\begin{tm}[Woodin]\label{tm:Woodin_nice_extension}
	Suppose $M_1^\#$ exists and is fully iterable.
	Let $\kappa_0$ be the least $M_1$-indiscernible.
	Then $M_1[\Sigma_{\kappa_0}]$ is a nice extension.
\end{tm}

We will show that this is optimal, in the following sense:

\begin{tm}\label{tm:nice_extensions}
	Suppose $M_1^\#$ exists and is fully iterable.
	 Let $\kappa_0$ be the least $M_1$-indiscernible.
	 Then:
	 \begin{enumerate}
	 	\item\label{item:kappa^+_nice} $M_1[\Sigma_{(\kappa_0^+)^M}]=M_1[\Sigma_{\kappa_0}]$,
	so this extension is nice.
	\item\label{item:kappa^++1_not_nice} There is a tree $\Tt\in M_1$ on $M_1$,
	with $\lh(\Tt)=(\kappa_0^+)^{M_1}$,
	such that $M_1^\#\in M_1[b]$ where $b=\Sigma_{M_1}(\Tt)$.
	In fact, we can take $\Tt$ definable over $M_1|(\kappa_0^+)^M$.
	\end{enumerate}
	\end{tm}

Nice extensions are formed by adding branches for trees
which are already present in $M_1$. But one can
construct, closing under $\Sigma_{M_1}$, and ask what is the least 
$\alpha\in\OR$ such that
\[ M_1^\#\in L_\alpha[M_1|\delta^{M_1},\Sigma_{M_1}]. \]
Woodin showed that $L_{\kappa_0}[M_1|\delta^{M_1},\Sigma_{M_1}]$ does not 
add bounded subsets of $\delta^{M_1}$ or kill the Woodinness of $\delta^{M_1}$.  Here we will improve this by showing:
\begin{tm}\label{tm:closure_under_Sigma}
		Suppose $M_1^\#$ exists and is fully iterable.
	Let $\kappa_0$ be the least $M_1$-indiscernible. Then
	 $L_{\kappa_0}[M_1|\delta^{M_1},\Sigma_{\kappa_0}]$
	is already closed under $\Sigma_{M_1}$, and therefore
\[ L_{\kappa_0}[M_1|\delta^{M_1},\Sigma_{M_1}]=
L_{\kappa_0}[M_1|\delta^{M_1},\Sigma_{\kappa_0}]\sub M_1[\Sigma_{\kappa_0}].\]
\end{tm}

After these main results, we also collect together a few other related
observations on iterates of $M_1$ and $L[x]$. In \S\ref{sec:kill_it}
we present an argument due to Schindler and the author, showing that $\sigma$-distributive forcing can destroy
the $(\om,\om_1+1)$-iterability of $M_1^\#$.

The ideas in this paper were developed in two main phases.
In the first phase,
some adaptation of Woodin's analysis of $\HOD^{L[x,G]}$
and  parts of the analysis  of the directed system $\mathscr{F}^{L[x]}_{\leq\delta}$, where $\delta=\om_1^{L[x]}$,
were worked out jointly by Steel and the author, in approximately 2010, as part of attempts to analyze $\HOD^{L[x]}$ (in particular, versions of
Claim \ref{clm:F_dense} of the proof of Theorem \ref{tm:delta-cc_mantle} were worked out),
and Steel showed that the resulting $M_\infty[\Sigma_\infty]$  is \emph{not} $\HOD^{L[x]}$,
via Claim \ref{clm:delta_infty>delta+} of that proof. At that time, Steel
and the author also
analyzed $\HOD^{L(\RR^+)}$ for certain symmetric extensions $\RR^+$ of $\RR\inter L[x]$
smaller than the traditional one, and this also involved
more local versions of Claim \ref{clm:F_dense}.

The second phase of the work was certainly motivated
by other recent work of Fuchs and Schindler \cite{vm0},
Sargsyan and Schindler \cite{vm1}, and Sargsyan, the author and Schindler \cite{vm2_v2}, analysing amongst other things mantles of mice.
Moreover, a forcing construction introduced in that context by Schindler
is exploited heavily in the proofs of Theorems \ref{tm:unctbl_cof} and \ref{tm:delta-cc_mantle}.
And the analysis of the $\kappa_0$-mantle of $M_{\swsw}$ mentioned 
earlier (see \cite{vm2_v2})
is in particular motivating and analogous to the methods here.
The new arguments in the paper from this phase were found in 2019--21.

The material in \S\ref{sec:background} is just a slight generalization of  Woodin's original analysis
and some related facts which are by now fairly well-known
and have already appeared in one form or another.

\section{Background: Directed systems based on $M_1$}\label{sec:background}

We assume the reader is familiar with the basics of inner model theory
as covered in \cite{outline}. We use notation as there
and as in \cite{premouse_inheriting}.
The main
technique we deal with is Woodin's analysis of $\HOD^{L[x,G]}$,
basically as developed in \cite{hod_as_core_model},
augmented by the calculations in \cite{Theta_Woodin_in_HOD}, and the forcing
due to Schindler \cite{vm1}.

In this section we will develop  this background material
in some detail, for self-containment and because
we need to generalize it a little. (The generalization
itself takes very little thought, however.)
So the reader might 
not need 
prior knowledge of the references just mentioned,
although they give more information.
The material in this section is essentially all standard,
with very slight adaptations to our settings.

 We say that an iteration tree is \emph{normal}
 if it is $\om$-maximal in the sense of \cite{fsit}.
 An iteration tree is \emph{terminally-non-dropping}
 iff $b^\Tt=[0,\infty]^\Tt$ exists and does not drop
 (we will only apply this terminology in the case that 
 $\Tt$ is a tree on a proper class model).
 
  A premouse $N$ is \emph{pre-$M_1$-like}
 iff $N$ is proper class, 1-small, has a (unique) Woodin cardinal.
Let $N$ be pre-$M_1$-like. Let $\Tt$ be a normal 
tree on $N$ of limit length. Then $\Tt$ is \emph{short}
iff there is $\alpha\in\OR$ such that
either $\J_\alpha(M(\Tt))$ projects ${<\delta(\Tt)}$
or $\J_\alpha(M(\Tt))$ is a Q-structure for $\delta(\Tt)$;
otherwise $\Tt$ is \emph{maximal} (that is,
$L[M(\Tt)]$ is a premouse satisfying ``$\delta(\Tt)$ is Woodin'').
We say that $N$ is \emph{normally-short-tree-iterable}
iff $N$ has a partial iteration strategy whose domain
includes all short normal trees.
A \emph{pseudo-normal iterate} of $N$
is a premouse $P$ such that either $P$ is a normal iterate of $N$
or there is a maximal normal tree $\Tt$ on $N$
such that $P=L[M(\Tt)]$.
A \emph{non-dropping pseudo-normal iterate}
is a pseudo-normal iterate which is proper class (hence pre-$M_1$-like).
A \emph{relevant finite pseudo-stack}
on $N$ is a sequence $\left<\Tt_\alpha\right>_{\alpha<n}$
for some $n<\om$ such that
there is a sequence $\left<N_\alpha\right>_{\alpha<n}$
where $N_0=N$, and for $\alpha<n$,
$N_\alpha$ is pre-$M_1$-like and
$\Tt_\alpha$ is a normal tree on $N_\alpha$,
and if $\alpha+1<n$
then either $\Tt_\alpha$ has successor length
and is terminally-non-dropping
and $N_{\alpha+1}=M_\infty^{\Tt_\alpha}$,
or $\Tt_\alpha$ is maximal and $N_{\alpha+1}=L[M(\Tt_\alpha)]$.
A \emph{non-dropping pseudo-iterate} of $N$
is a premouse $P$ such that there is a relevant
finite pseudo-stack on $N$ with last model $P$
(the trees in the stack are allowed to be trivial,
i.e. use no extenders), and $P$ is proper class (hence pre-$M_1$-like).
A \emph{relevant stack} on $N$ is a sequence 
$\left<\Tt_\alpha\right>_{\alpha<\lambda}$ such that
there is a sequence $\left<N_\alpha\right>_{\alpha<\lambda}$,
such that $\lambda\in\OR$,
 $N_0=N$, for each $\alpha<\lambda$, $N_\alpha$
 is pre-$M_1$-like and $\Tt_\alpha$ is a normal tree on $N_\alpha$,
 and if $\alpha+1<\lambda$ then $\Tt_\alpha$  has successor length,
  $b^{\Tt_\alpha}$ is non-dropping and
$N_{\alpha+1}=M^{\Tt_\alpha}_\infty$,
 and if $\eta<\lambda$ is a limit then
 $N_\eta$ is the direct limit of the $N_\alpha$ for $\alpha<\eta$
 under the iteration maps.

 A pre-$M_1$-like premouse $N$ is \emph{short-tree-iterable}
 iff there is an iteration strategy for $N$
 for relevant finite pseudo-stacks with normal
 components of arbitrary length.
 
 A premouse $N$ is \emph{$M_1$-like} iff
 \begin{enumerate}[label=(\roman*)]
  \item  $N$ is pre-$M_1$-like
 and short-tree-iterable, and
 \item\label{item:fullnorm}  every non-dropping pseudo-iterate
 of $N$ is in fact a non-dropping pseudo-normal-iterate of $N$.
 \end{enumerate}
By \cite{fullnorm}, every non-dropping iterate of $M_1$ is $M_1$-like.

 Let $\vec{\OR}=[\OR]^{<\om}\cut\{\emptyset\}$.
 Let $s\in\vec{\OR}$. Then 
$s^-=s\cut\{\max(s)\}$. 
Let $N$ be $M_1$-like with $\delta^N<\max(s)$.
 Then
 \[ \gamma^N_s=\sup(\Hull^{N|\max(s)}(\{s^-\})\inter\delta^N)\]
and
\[ H^N_s=\Hull^{N|\max(s)}(\gamma^N_s\cup\{s^-\}); \]
(note the hulls above are uncollapsed; it follows that 
$H^N_s\inter\delta^N=\gamma^N_s$).

For $M_1$-like premice $P,Q$, say $P\dashrightarrow Q$ iff
$Q$ is a pseudo-normal iterate of $P$.
Note that $\dashrightarrow$ is a partial order
(using condition \ref{item:fullnorm} above for transitivity), and if $P\dashrightarrow Q$ then
there is a unique normal tree $\Tt_{PQ}$ leading from $P$ to $Q$
such that $\Tt_{PQ}$ has no maximal proper segment.
If $P$ happens to be an iterate of $M_1$,
then let $i_{PQ}:P\to Q$ be the iteration map.
We say a set $\mathscr{D}$ of $M_1$-like premice
is \emph{directed} iff for all $P,Q\in\mathscr{D}$,
there is $R\in\mathscr{D}$ such that $P\dashrightarrow R$ and $Q\dashrightarrow R$,
and \emph{pretty good} if it is directed and some iterate of $M_1$
is in $\mathscr{D}$.

 Let $\mathscr{D}$ be a pretty good set of premice.
Let $\mathscr{F}=\{N\in\mathscr{D}\bigm|N\text{ is an iterate of }M_1\}$.
Note that $\mathscr{F}$ is directed and is dense in $\mathscr{D}$
with respect to $\dashrightarrow$,
and the iteration maps $i_{PQ}$ for $P,Q\in\mathscr{F}$
(with $P\dashrightarrow Q$) are commuting:
\[ P,Q,R\in\mathscr{F}\text{ and }P\dashrightarrow Q\dashrightarrow R\implies i_{PR}=i_{QR}\com 
i_{PQ}.\]
 Let
 \[ M_\infty^{\mathscr{F}}=\dirlim\left<P,Q;i_{PQ}\bigm|Q,P\in\mathscr{F}\text{ 
with }P\dashrightarrow Q\right>, \]
and let $i_{P\infty}$ be the direct limit map.
Then by \cite{fullnorm}, $M_\infty$
is in fact a normal iterate of $M_1$ and of each $P\in\mathscr{F}$
and $i_{P\infty}$ is just the iteration map $i_{PM_\infty}$.

Still with $\mathscr{D}$ pretty good, let $N_0\in\mathscr{D}$ and 
$s\in\vec{\OR}$.
 Then $N_0$ is \emph{$(s,\mathscr{D})$-iterable}
 iff for all $N_1,N_2,N_3\in\mathscr{D}$ with $N_0\dashrightarrow N_1\dashrightarrow N_2\dashrightarrow N_3$,
 letting $\Tt_{ij}=\Tt_{N_iN_j}$,
 $\Coll(\om,\max(s))$ forces
\begin{enumerate}
 \item  there is a $\Tt_{12}$-cofinal branch
 $b$ which fixes $s$ (meaning 
  $\max(s)\in\wfp(M^{\Tt_{12}}_b)$, $b$ does not drop, 
 and $i^{\Tt_{12}}_b(s)=s$), and
\item  whenever $b_{12},b_{23},b_{13}$ are $\Tt_{12},\Tt_{23},\Tt_{13}$-cofinal 
branches
respectively which fix $s$, we have
\[ i^{\Tt_{13}}_{b_{13}}\rest \gamma^{N_1}_s=i^{\Tt_{23}}_{b_{23}}\com 
i^{\Tt_{12}}_{b_{12}}\rest\gamma^{N_1}_s,\]
and hence
$i^{\Tt_{13}}_{b_{13}}\rest H^{N_1}_s=i^{\Tt_{23}}_{b_{23}}\com 
i^{\Tt_{12}}_{b_{12}}\rest H^{N_1}_s$.
\end{enumerate}

For $P\in\mathscr{F}$, say that $P$ is \emph{$(s,\mathscr{D})$-stable}
(or \emph{$(s,\mathscr{F})$-stable}) iff for all $Q\in\mathscr{D}$
with $P\dashrightarrow Q$ (hence $Q\in\mathscr{F}$)
we have $i_{PQ}(s)=s$. The $(s,\mathscr{D})$-stable $P$ are dense
in $\mathscr{D}$ and if
$P\in\mathscr{F}$ is $(s,\mathscr{D})$-stable
then $P$ is $(s,\mathscr{D})$-iterable.

The iteration maps $i_{PQ}$ used above are not in general
classes of $L[x]$ (considering here the case that $P,Q\in L[x]$). We now define a
\emph{covering system} $\widetilde{\mathscr{D}}$
for $\mathscr{F}$, which will be definable from $\mathscr{D}$ over any proper
class model $\mathscr{L}$ of ZF with $\mathscr{D}\in\mathscr{L}$.
Let $\widetilde{\mathscr{D}}$
be the class of pairs $(P,s)$ such that $P\in\mathscr{D}$
and $s\in\vec{\OR}$ with $\delta^P<\max(s)$
and $P$ is $(s,\mathscr{D})$-iterable.
For $(P,s),(Q,t)\in\widetilde{\mathscr{D}}$,
say $(P,s)\leq(Q,t)$
iff $P\dashrightarrow Q$ and $t\supseteq s$. This is a partial order.
Let $(P,s)\in\widetilde{\mathscr{D}}$ index
the structure $H^P_s$.
For $(P,s)\leq(Q,t)$ we have the $\Sigma_0$-elementary embedding
\[ i^{PQ}_{st}:H^P_s\to H^Q_t, \]
given by
\[ i^{PQ}_{st}=i^{\Tt_{PQ}}_b\rest H^P_s:H^P_s\to H^Q_s\elem_0 H^Q_t \]
where $b$ is any generic $\Tt_{PQ}$-cofinal branch fixing $s$,
as provided for by $(s,\mathscr{D})$-iterability.
(Note this map is independent of the choice of $b$,
and independent of $t$.)
Note that if $(P,s),(Q,t),(R,u)\in\widetilde{\mathscr{D}}$
and $(P,s)\leq(Q,t)\leq(R,u)$
then
\[ i^{PR}_{su}=i^{QR}_{tu}\com i^{PQ}_{st}.\]
And $\widetilde{\mathscr{D}}$
is directed under $\leq$, because if $(P,s),(Q,t)\in\widetilde{\mathscr{D}}$
then there is $N\in\mathscr{F}$ which is $(s\cup t,\mathscr{F})$-stable,
and then letting $R\in\mathscr{F}$ with $N,P,Q\dashrightarrow R$, we have that
$R$ is $(s\cup t,\mathscr{F})$-stable, hence
$(s\cup t,\mathscr{D})$-iterable,
so $(R,s\cup t)\in\widetilde{\mathscr{D}}$,
and $(P,s),(Q,t)\leq (R,s\cup t)$.

So we can define
\[M_\infty^{\widetilde{\mathscr{D}}}=\dirlim\left<H^P_s,H^Q_t;i^{PQ}_{st}
\bigm| (P,s),(Q,t)\in\widetilde{\mathscr{D}}\text{ and }(P,s)\leq 
(Q,t)\right>,\]
and let $\widetilde{i}_{Ps,\infty}:H^P_s\to 
M_\infty^{\widetilde{\mathscr{D}}}$
be the ($\Sigma_0$-elementary) direct limit map. Given
$\mathscr{L}\sats\ZF$ is a transitive proper class
with $\mathscr{D}\in\mathscr{L}$, then note that
over $\mathscr{L}$, from the parameter  $\mathscr{D}$, we can define
 $\widetilde{\mathscr{D}}$ and
the associated ordering $\leq$,
$H^P_s,\widetilde{i}_{Ps,Qt}\in\mathscr{L}$,
and the class functions $(P,s)\mapsto H^P_s$ and $((P,s),(Q,t))\mapsto\widetilde{i}_{Ps,Qt}$
are  $\mathscr{L}$-definable from $\mathscr{D}$, and therefore
so is $M^{\widetilde{\mathscr{D}}}_\infty$
and the system of direct limit maps. Moreover, this is uniform in $\mathscr{D},\mathscr{L}$.

Working in $V$, define
$\sigma:M_\infty^{\widetilde{\mathscr{D}}}\to M_\infty^{\mathscr{F}}$
as follows: given $(P,s)\in\widetilde{\mathscr{D}}$
and $x\in H^P_s$,
let $Q\in\mathscr{F}$ such that $P\dashrightarrow Q$
and $Q$ is $(s,\mathscr{F})$-stable, and set
\[ \sigma(\widetilde{i}_{Ps,\infty}(x))=i_{QM_\infty}(i_{Ps,Qs}(x)).\]
Note this is independent of the choice of $Q$,
and $\sigma$ is $\Sigma_0$-elementary
and $\in$-cofinal, and hence $\Sigma_1$-elementary.
Therefore $M_\infty^{\widetilde{\mathscr{D}}}$ is wellfounded
and models $\ZFC$ and $\sigma$ is fully elementary.
We next show that, in the circumstances of interest,
we have $M_\infty^{\widetilde{\mathscr{D}}}=M^{\mathscr{F}}_\infty$
and $\sigma=\id$.

For a set $A\sub\OR$ let $\left<\kappa_\alpha^A\right>_{\alpha\in\OR}$
be the increasing enumeration of the Silver indiscernibles
for $L[A]$ (these exist as $M_1^\#$ is fully iterable).
Also write $\mathscr{I}^A=\{\kappa^A_\alpha\}_{\alpha\in\OR}$.
So we have that $\mathscr{I}^A$ is the unique
$\mathscr{I}$ such that:
\begin{enumerate}
 \item $\mathscr{I}\sub\OR$ is club proper class,
 \item $\mathscr{I}$ consists of model-theoretic
 indiscernibles for the structure
 \[ (L[A],\in,A,\{c_\alpha\}_{\alpha<\sup(A)}),\]
  where $c_\alpha$ is a constant interpreted by $\alpha$,
 \item $L[A]=\Hull^{L[A]}_{\Sigma_1}(\mathscr{I}\cup\{A\}\cup(\sup A))$.
\end{enumerate}
For $\Gamma\sub\OR$ write $\vec{\Gamma}=[\Gamma]^{<\om}\cut\{\emptyset\}$.
Note that if $A\sub\om$ then
$L[A]=\bigcup_{s\in\vec{\mathscr{I}^A}}
 J^A_s$
where
\[ J^A_s=\Hull^{L_{\max(s)}[A]}(\{A,s^-\}).\]
Note this also holds if $A=(x,G)$ for some $x\sub\om$ and $G$
which is $(L[x],\PP)$-generic for some $\PP\in L_{\kappa_0^x}[x]$.
And for $M_1$-like $P$,
let $\kappa_\alpha^P=\kappa_\alpha^{P|\delta^P}$
and  $\mathscr{I}^P=\mathscr{I}^{P|\delta^P}$.
Then
$\delta^P\inter\Hull_{\Sigma_1}^P(\mathscr{I}^P)$ is cofinal in $\delta^P$ (letting $\delta$ be the supremum,
we have
\[ \delta=\delta^P\inter\Hull_{\Sigma_1}^P(\delta\cup\mathscr{I}^P),\]
and if $\delta<\delta^P$, then letting $\bar{P}$ be the
transitive collapse of this hull and $\pi:\bar{P}\to P$ the uncollapse,
then $\bar{P}|\delta=P|\delta$ and $\pi(\delta)=\delta^P$
and $\bar{P}\sats$``$\delta$ is Woodin'' and
and $\bar{P}$ is proper class, but as $\delta<\delta^P$ and $P$ is 1-small,
there is a Q-structure for $\delta$ in $\bar{P}$, a contradiction).
Note that therefore
$P$ has universe $\bigcup_{s\in\vec{\mathscr{I}^P}}H^P_s$.

Let $M_1\dashrightarrow P\dashrightarrow Q$.
Then $i_{M_1P}$ is continuous at all $M_1$-indiscernibles, 
which easily implies that
$\mathscr{I}^P=i_{M_1P}``\mathscr{I}^{M_1}$.
Therefore $\mathscr{I}^Q=i_{PQ}``\mathscr{I}^P$.

For an $M_1$-like $P$, we write $\BB^P$ for the $\delta^P$-generator extender
algebra of $P$ at $\delta^P$, constructed using extenders
$E\in\es^P$ such that $\nu_E$ is a $P$-cardinal.

Fix a real $x$ and $\kappa=\kappa_0^x$.
Say that $\mathscr{D}$ is \emph{$x$-good} (or just \emph{good},
if $x$ is fixed)
iff there are $\eta,G,\widehat{\mathscr{F}},\left<\mathscr{F}^P\right>_{P\in\widehat{\mathscr{F}}}$ such that:
\begin{enumerate}
\item either:
\begin{enumerate}[label=--]
	\item  $\eta=G=\emptyset$ and
 $\mathscr{D}\in L[x]$ and $\mathscr{D}\sub L_{\kappa}[x]$
(hence $\delta^N<\kappa$ for each $N\in\mathscr{D}$),
or \item $\eta\leq\kappa$ is a limit cardinal of $L[x]$ and $G$ is $(L[x],\Coll(\om,{<\eta}))$-generic
and $\mathscr{D}$ is the set of all $M_1$-like premice
$N$ with $N|\delta^N\in L_\eta[x,G]$.
\end{enumerate}

 \item  $\widehat{\mathscr{F}}\sub\mathscr{F}$ and $(\widehat{\mathscr{F}},\dashrightarrow)$ is cofinal
 in $(\mathscr{F},\dashrightarrow)$,
 \item  for all $P\in\widehat{\mathscr{F}}$:
 \begin{enumerate}[label=--]
  \item $x$ is $(P,\BB^P)$-generic (so $L[x]\sub P[x]\sub L[x,G\rest\gamma]$
  for some $\gamma<\eta$),
  \item $\mathscr{F}^P\in P$,
  \item  $P\in\mathscr{F}^P\sub\{Q\in P\bigm| P\dashrightarrow Q\}$,
 \item  $\mathscr{F}$ and $\mathscr{F}^P$ are 
cofinal,
  meaning that for every $Q\in\mathscr{F}^P$
  there is $R\in\mathscr{F}$ with $Q\dashrightarrow R$,
  and vice versa
 \item $P|\delta^P$ is $(Q,\BB^{Q})$-generic
  for every $Q\in\mathscr{F}^P$,
  \item $i_{PQ}(\eta)=\eta$ and
 $i_{PQ}(\mathscr{F}^P)=\mathscr{F}^Q$ for all 
$Q\in\widehat{\mathscr{F}}$
  with $P\dashrightarrow Q$.
 \end{enumerate}
\end{enumerate}

From now on, we assume that $\mathscr{D}$ is $x$-good,
witnessed by objects as above.
Let $P\in\widehat{\mathscr{F}}$
and let $Q,R\in\mathscr{F}$
with $P\dashrightarrow Q\dashrightarrow R$.
We claim:
\begin{enumerate}[label=--]
 \item $\mathscr{I}^P=\mathscr{I}^Q=\mathscr{I}^R=\mathscr{I}^x$,
\item $i_{QR}\rest\mathscr{I}^x=\id$, and therefore
\item $Q$ is $(s,\mathscr{F})$-stable and $(s,\mathscr{D})$-iterable
for each $s\in\vec{\mathscr{I}^x}$.
\end{enumerate}
For by considering the $\BB^P$-forcing relation in $P$
and since $\BB^P\sub\delta^P$ and
$L[x]=\Hull^{L[x]}_{\Sigma_1}(\mathscr{I}^x\cup\{x\})$,
it follows that
\[ P=\Hull_{\Sigma_1}^P(\mathscr{I}^x\cup\delta^P). \]
Let $\mathscr{I}'=\mathscr{I}^P\inter\mathscr{I}^x$,
which is a proper class club. Therefore 
\[ \kappa\sub\Hull^{L[x]}_{\Sigma_1}(\mathscr{I}'\cup\{x\}).\]
Since $P|\delta^P\in L[x,G]|\kappa$,
it easily follows that $\mathscr{I}^x$ are (a club class of)
model-theoretic indiscernibles for 
$(P,\in,\es^P,\{c_\alpha\}_{\alpha<\delta^P})$.
So by uniqueness, we get $\mathscr{I}^P=\mathscr{I}^x$.
Now consider $Q,R$. Let $R\dashrightarrow S\in\widehat{\mathscr{F}}$.
Then $x$ is $\BB^S$-generic
over $S$, so $\mathscr{I}^S=\mathscr{I}^x=\mathscr{I}^P$.
But $i_{PS}``\mathscr{I}^P=\mathscr{I}^S$,
so $i_{PS}\rest\mathscr{I}^P=\id$.
By commutativity, therefore 
$\mathscr{I}^Q=\mathscr{I}^R=\mathscr{I}^P=\mathscr{I}^S$
and $i_{PQ}\rest\mathscr{I}^P= i_{QR}\rest\mathscr{I}^Q=i_{RS}\rest\mathscr{I}^R=\id$.

We now claim that
$M_\infty^{\widetilde{\mathscr{D}}}= M_\infty^{\mathscr{F}}
\text{ and }\sigma=\id$.
For  this, we just need to see that $\sigma$ is surjective.
So let $x\in M_\infty^{\mathscr{F}}$.
Let $P\in\mathscr{F}$ and $\bar{x}\in P$ with
$i_{P\infty}(\bar{x})=x$.
Fix $s\in\vec{\mathscr{I}^P}$ such that
$\bar{x}\in H^P_s$. By the claim  above,
$P$ is $(s,\mathscr{F})$-stable and
$(P,s)\in\widetilde{\mathscr{D}}$,
and therefore
$\sigma(i_{Ps,\infty}(\bar{x}))=i_{P\infty}(\bar{x})=x$.
So $\sigma$ is surjective, as desired.

Now for $\alpha\in\OR$ and $(s,P)\in\widetilde{\mathscr{D}}$
with $\alpha\in s^-$ let
\[ \alpha^*=i_{Ps,\infty}(\alpha),\]
and note $\alpha\mapsto\alpha^*$ is $\mathscr{L}$-definable from $\mathscr{D}$.
Note also that 
$\alpha^*=
\min_{P\in\mathscr{F}}i_{P\infty}(\alpha)$.

Let $P\in\widehat{\mathscr{F}}$. Working in $V$,
let $M_\infty^{\Fff^P}$ be the direct limit of $\Fff^P$
under the iteration maps. By the cofinality of $\Fff$ and $\Fff^P$,
we have $M_\infty^{\Fff^P}=M_\infty^{\Fff}$,
and the associated direct limit maps $P\to M^{\Fff^P}_\infty$ and $P\to M^{\Fff}_\infty$ are
just the iteration map.
Now work in $P$. Define $\widetilde{\mathscr{F}}^P$
from $\mathscr{F}^P$, together with its
associated structures $H^P_s$ and maps $\widetilde{i}_{Ps,Qt}$,
just as 
$\widetilde{\mathscr{D}}$
is defined from $\mathscr{D}$, etc, and define
$M_\infty^{\widetilde{\mathscr{F}^P}}$
as its direct limit. Let
\[ \sigma^P:M_\infty^{\widetilde{\mathscr{F}^P}}\to 
M^{\Fff^P}_\infty \]
be as before.
Then  as before, 
$M_\infty^{\widetilde{\mathscr{F}^P}}=M_\infty^{\mathscr{F}^P}
=M_\infty^{\Fff}$ and  $\sigma^P=\id$. Likewise with $*^P$ defined in $P$
from $\widetilde{\Fff^P}$ just as $*$ was defined above, we have
 $*^P=*$.  It follows that for $P,Q\in\widehat{\Fff}$
 with $P\dashrightarrow$, we have $M_\infty^{\mathscr{F}},*$ are classes of $P$
 and $i_{PQ}((M_\infty^{\mathscr{F}},*))=(M_\infty^{\mathscr{F}},*)$.

Let $N=M_\infty^{\mathscr{F}}$.
Let $\mathscr{F}^{N}=i_{P\infty}(\mathscr{F}^P)$
for any $P\in\widehat{\mathscr{F}}$ (note this is independent
of such $P$), and likewise $\widetilde{\Fff^N}$.
Note that 
\[ 
\mathscr{I}^{N}=i_{M_1N}``\mathscr{I}^{M_1}=
\{\mu^*\bigm|\mu\in\mathscr{I}^x\}.\]
(We might not have $\mathscr{I}^{M_1}=\mathscr{I}^x$;
for example, if $M_1^\#\in L[x]$ then $\kappa_0^{M_1}<\om_1^{L[x]}$.)
By $x$-goodness,
every $M\in\mathscr{F}^N$ is such that
$N\dashrightarrow M$ and
$\delta^{M}<\kappa_0^{N}=(\kappa_0^x)^*$
and $N|\delta^N$
is $(M,\BB^{M})$-generic. So calculations
with indiscernibles like before give that 
$i_{NM}\rest\mathscr{I}^{N}=\id$. It follows
that $M_\infty^{\widetilde{\Fff^N}}=M_\infty^{\Fff^N}$
and the natural map $\sigma^N:M_\infty^{\widetilde{\Fff^N}}\to 
M_\infty^{\Fff^N}$
(defined as before) is just $\sigma^N=\id$.

Let $k=i_{NM_\infty^{\Fff^N}}:N\to M_\infty^{\Fff^N}$ be the iteration map.
We claim  $\alpha^*=k(\alpha)$ for all $\alpha\in\OR$.
For let $P\in\widehat{\Fff}$ be $(\alpha,\Fff)$-stable,
with $\alpha\in\rg(i_{P\infty})$, and let $\bar{\alpha}\in P$
with $i_{P\infty}(\bar{\alpha})=\alpha$,
and let
$s\in\vec{\mathscr{I}^x}=\vec{\mathscr{I}^P}$
with $\bar{\alpha},\alpha\in H^P_s$. Then $P\sats$``$(s,P)\in\widetilde{\Fff^P}$
and $\alpha=\widetilde{i}_{Ps,\infty}(\bar{\alpha})$'',
so $N=M_\infty\sats$``$(s^*,N)\in\widetilde{\Fff^N}$
and $\alpha^*=\widetilde{i}_{Ns^*}(\alpha)=\alpha^*$''
(where $s^*,\alpha^*$ are just parameters), but since $s^*\in\mathscr{I}^{N}$ 
 and $N$ is $(s^*,\mathscr{F}^{N})$-stable,
 therefore 
$k(\alpha)=\widetilde{i}_{Ns^*}(\alpha)=\alpha^*$.
 
So
 $M_\infty[*]=L[\es^{M_\infty},*]=L[\es^{M_\infty},k]$,
 and moreover,
 $M_\infty[*]=M_\infty[\Sigma_\infty]$,
 where $\Sigma_\infty=\Sigma_{M_1}\rest\Fff^N$,
 i.e. $\Sigma_\infty$ encodes the correct branches through all trees in $\Fff^N$
 (in a canonical manner, without encoding extra information).
 For we clearly have $M_\infty[*]\sub M_\infty[\Sigma_\infty]$. 
 The converse is a consequence of the following fact. Let $M_1\dashrightarrow
 P\dashrightarrow Q\dashrightarrow R$,
 and suppose we know $\ell=i_{PR}$, and want
 to compute $i_{PQ}$ and $i_{QR}$ from $\ell,P,Q,R$.
 Let $H$ be $(V,\Coll(\om,\alpha))$-generic
 for some $\alpha\in\OR$,
 and let $(b,c)\in V[H]$ be such that $b$
 is a  $\Tt=\Tt_{PQ}$-cofinal branch
 and $c$ a $\Uu=\Tt_{QR}$-cofinal branch
 with $i^\Uu_c\com i^\Tt_b\rest P|\delta^P\sub i_{PR}$.
 We claim then that $b=\Sigma_{M_1}(\Tt)$
 and $c=\Sigma_{M_1}(\Uu)$, and hence we can compute $b,c$
 from the given data, as desired. For if
 $c\neq d=\Sigma_{M_1}(\Uu)$,
 then by the Zipper Lemma, $\rg(i^\Uu_c)\inter\rg(i^\Uu_d)$
 is bounded in $\delta(\Uu)$,
 although this is false, because $i^\Uu_c\com i^\Tt_b\sub i_{PR}$.
 So $c=\Sigma_{M_1}(\Uu)$.
 Now suppose $b\neq d=\Sigma_{M_1}(\Tt)$.
 Then $\rg(i^\Tt_b)\inter\rg(i^\Tt_d)$ is bounded in $\delta(\Tt)$,
 but then $\rg(i^\Uu_c\com i^\Tt_b)\inter\rg(i^\Uu_c\com i^\Tt_d)$
 is bounded in $\delta(\Uu)$, which is again false, completing the proof.
 
 Let $\delta_\infty=\delta^{M_\infty}$.
 We claim
$V_{\delta_\infty}^{M_\infty[*]}=V_{\delta_\infty}^{M_\infty}$.
For let $X\sub\alpha<\delta_\infty$ with $X\in M_\infty[*]$.
Then there is $\eta\in\OR$ and a formula
$\varphi$ such that
$\beta\in X$ iff $M_\infty[*]\sats\varphi(\eta,\beta)$,
for each $\beta<\alpha$. It follows
that there is $P\in\widehat{\Fff}$ such that
$i_{PQ}(X)=X$ for all $Q\in\widehat{\Fff}$
with $P\dashrightarrow Q$, and $\alpha\in\rg(i_{P\infty})$. Fix such a $P$, and let 
$s\in\vec{\mathscr{I}^x}$
with $\alpha<\sup i_{P\infty}``\gamma^P_s$.
Let $\bar{\alpha}<\gamma^P_s$ with $i_{P\infty}(\bar{\alpha})=\alpha$.
In $P$, define
\[ 
\bar{X}=\{\bar{\beta}<\bar{\alpha}\bigm|\widetilde{i}_{Ps,\infty}(\bar{\beta}
)\in X\}.\]
Then $i_{P\infty}(\bar{X})=X$. 
For if $\beta\in\alpha\inter\rg(i_{P\infty})$,
then clearly $\beta\in X$ iff $\beta\in i_{P\infty}(\bar{X})$.
And for other $\beta<\alpha$,
we can find some $Q\in\widehat{\Fff}$
with $P\dashrightarrow Q$ and $\beta\in\rg(i_{Q\infty})$,
and then because $i_{PQ}(X)=X$,
we still get that $\beta\in X$ iff $\beta\in 
i_{P\infty}(\bar{X})=i_{Q\infty}(i_{PQ}(\bar{X}))$.
This proves the claim.

Let $M_1\dashrightarrow M\dashrightarrow P\in\widehat{\Fff}$
with
$\mathscr{F}^P\in\rg(i_{MP})$.
Let $j=i_{MM_\infty}$ and 
 $H=\Hull_{1}^{M_\infty[*]}(\rg(j))$.
We claim 
$H\inter M_\infty=\rg(j)$.
For let $\alpha\in\OR\inter H$;
we must see that $\alpha\in\rg(j)$.
Now $\alpha$ is definable over $M_\infty[*]$
from $j(t,\gamma)$ for some $t\in[\mathscr{I}^{M}]^{<\om}$ and $\gamma<\delta^M$.
Note then that $\alpha$ is definable over $P$ from $i_{MP}(t,\gamma)$,
uniformly in $P\in\widehat{\Fff}$ with $M\dashrightarrow P$.
It follows that $\alpha^*=k(\alpha)$ is
definable over $N=M_\infty$ from $j(t,\gamma)$,
recalling that $k:N\to M_\infty^{\Fff^N}$ is the iteration map,
and so $k(\alpha)\in\rg(j)$.
But there is $u\in\vec{\mathscr{I}^N}$
with $\alpha\in H^N_u$, and note that $\widetilde{i}_{Nu,\infty}\in\rg(j)$
and $\widetilde{i}_{Nu,\infty}=k\rest H^N_u$.
But since $\alpha^*=k(\alpha)\in\rg(j)$, therefore $\alpha\in\rg(j)$ also.

Letting $M^+$ be the transitive collapse of $H$
and $j^+:M^+\to M_\infty^{\Fff}[*]$ the uncollapse map,
it follows that $M\sub M^+$ and $j\sub j^+$,
and $V_{\delta^M}^M=V_{\delta^M}^{M^+}$.

We now claim that $\delta_\infty$ is Woodin
in $M_\infty^{\Fff}[*]$ (equivalently,
$\delta^M$ is Woodin in $M^+$). Suppose not and let $A\in\pow(\delta_\infty)\inter M_\infty^{\Fff}[*]$ be a counterexample to the Woodinness of $\delta_\infty$
there; note that we can take $A\in\rg(j^+)$. 
Write $j^+(A^M)=A$. Let $P\in\widehat{\Fff}$ with $M\dashrightarrow P$
and $M\neq P$. Define $P^+$ and $i^+:P^+\to M_\infty^\Fff$ analogously
to $M^+$ and $j^+$, and $h^+:M^+\to P^+$ such that
$i^+\com h^+=j^+$. Consider the amenable structure $(M|\delta^M,A^M)$.
Note that we can convert
$\Tt_{MM_\infty}\conc b$ (where $b$ is the correct $\Tt_{MM_\infty}$-cofinal branch)
into a tree $\Tt_{MM_\infty}'\conc b$ on $(M|\delta^M,A^M)$ with last model
$(M_\infty|\delta_\infty,A)$ and iteration map $j\rest M|\delta^M$.
Moreover, since $\Tt_{MM\infty}$ is normal,
a slight variant of the proof of the Zipper Lemma shows that
$(b,A^M)$ is the unique pair $(c,C)$ such that this works,
in any universe which contains such a $(c,C)$. By absoluteness
and homogeneity of collapse forcing,
it follows that $(b,A^M)\in L[x,G]$,
and hence $L[x,G]$ computes $j$. Similarly, $L[x,G]$ computes $i$
and $h$. But $P\in\widehat{\mathscr{F}}$,
so $P$ is a set-ground of $\mathscr{L}$, but by \cite{gen_kunen_incon},
there can be no transitive proper class $M'$
and elementary $h':M'\to P$
such that $(M',h')$ is a class of $L[x,G]$,
 a contradiction.

 We now show that $M_\infty[*]$ is a ground
 of $L[x]$ for a forcing $\CC\sub\delta_\infty$
 which has the $\delta_\infty$-cc in $M_\infty[*]$;
 the argument here is  due to Schindler (though in some cases
 one can use more traditional methods using Vopenka).
 (We define the forcing slightly differently to how Schindler
 does, but it is equivalent.)
 Writing $N=M_\infty$,
 we have  $k:N\to M_\infty^{\Fff^N}$ is a class of $M_\infty[*]$.
 Working in $M_\infty[*]$, let $\mathscr{L}$ be the proper class infinitary Boolean language, given by a starting with a collection
 $\{v_n\}_{n<\om}$
 of propositional variables, and closing under negation
 and arbitrary set-sized disjunctions.
 Then define $\CC$ to be the subalgebra of the $\om$-generator extender algebra
 $\BB=\BB^{M_\infty}_\om$ of $M_\infty$ with conditions
 \[ \CC=\{\|k(\varphi)\|_{\BB}\bigm|\varphi\in\mathscr{L}\}, \]
 where $\|\psi\|_{\BB}$ denotes the Boolean value of a statement
 with respect to $\BB$, and we interpret
 $\left<v_n\right>_{n<\om}$ as the generic real for $\BB$. Since $V_{\delta_\infty}^{M_\infty}=V_{\delta_\infty}^{M_\infty[*]}$ and $\delta_\infty$ is Woodin
 in $M_\infty[*]$, $\BB$ is a complete Boolean algebra
 in $M_\infty[*]$ with the $\delta_\infty$-cc there, so this all makes sense,
 and for $\psi=k(\varphi)$, since $\psi\in M_\infty$,
 we have $(\|\psi\|_\BB)^{M_\infty}=(\|\psi\|_\BB)^{M_\infty[*]}$.
 It easily follows that $\CC$ is also a complete Boolean
 algebra with the $\delta_\infty$-cc in $M_\infty[*]$,
 and $\CC\sub\delta_\infty$. We claim that $x$ is $M_\infty[*]$-generic
 for $\CC$, with generic filter
 \[ G_x=\{\|k(\varphi)\|_\BB\bigm| \varphi\in\mathscr{L}\wedge x\sats\varphi\}. \]
 For $G_x$ is easily a filter. For genericity,
 let $\left<\varphi_\alpha\right>_{\alpha<\lambda}\in M_\infty[*]$ be
 such that $\left<\|k(\varphi_\alpha)\|_\BB\right>_{\alpha<\lambda}$
 is a maximal
 antichain of $\CC$. So $\lambda<\delta_\infty$.
 Let $\psi=\bigvee_{\alpha<\lambda}\varphi_\alpha$, and note that
 $\|\psi\|_{\BB}=\bigvee_{\alpha<\lambda}\|\varphi_\alpha\|_{\BB}$.
 We want to see $x\sats\psi$.
Suppose $x\sats\neg\psi$.
Let $P\in\widehat{\Fff}$ with $P$ being $(\psi,\mathscr{D})$-stable;
that is, $P$ is $(\alpha_\psi,\mathscr{D})$-stable where $\alpha_\psi$ is the rank of $\psi$ in the order of constructibility of $M_\infty[*]$.
Since $x$ is extender algebra generic over $P$ and $\psi\in P$,
if $x\sats\neg\psi$ then
$P\sats\|\neg\psi\|_{\BB^P}\neq 0$,
and therefore $M_\infty\sats\|k(\neg\psi)\|_{\BB}\neq 0$,
and so $\|k(\neg\psi)\|_{\BB}$ is a non-zero condition in $\CC$.
But it is easy to see that $\|k(\neg\psi)\|_{\BB}\incompat \|k(\varphi_\alpha)\|_{\BB}$ for each $\alpha<\lambda$, contradicting maximality.
 
 As an immediate corollary, $\delta_\infty$ is a regular cardinal
 in $M_\infty[*][x]$, and hence also in $L[x]$.
 
***To add: $M_\infty[\Sigma_\infty]$ is fully iterable (but
the proof is standard; see \cite{hod_as_core_model}
or \cite{Theta_Woodin_in_HOD}).

Now fix a limit cardinal $\eta_0$ of 
$L[x]$
 with $\eta_0\leq\kappa_0^x$ and $M_1|\delta^{M_1}\in L[x]|\eta_0$.
Let $G_0$ be $(L[x],\Coll(\om,{<\eta_0}))$-generic, and
 $\HC^+=\bigcup_{\alpha<\eta_0}\HC^{L[x,G_0\rest\alpha]}$.
The usual homogeneity argument shows that $\HC^+=\HC^{L(\HC^+)}$,
and in fact,
\[ \HC^+=\HC^{\HOD^{L[x,G_0]}_{\HC^+}}.\]
Let $\RR^+=\HC^+\inter\RR=L(\HC^+)\inter\RR=L(\RR^+)\inter\RR$.

Write $\mathscr{D}^{G_0}$ for the set of $M_1$-like
premice $N$ with $N|\delta^N\in\HC^+$.
Let $\mathscr{D}=\mathscr{D}^{G_0}$ and $\mathscr{F}=\mathscr{F}_{\mathscr{D}}$.
Let us show that $\mathscr{D}$ is $x$-good.
Note that $\mathscr{D},\mathscr{F}$ are closed under psuedo-comparison
and pseudo-$x$-genericity iteration.
Note that for each $N\in\mathscr{F}$
with $x$ extender algebra generic
over $N$, there is $G'$ which is $(N,\Coll(\om,{<\eta_0}))$-generic
such that $N[G']=L[x,G_0]$. Also, for such $N$,
$\delta^N$ is a regular cardinal of $L[x]$
with $\delta^N<\eta_0$, and since there is $\alpha<\eta_0$
with $N|\delta^N\in L[x,G_0\rest\alpha]$,
we have $L[x]\sub N[x]\sub L[x,G_0\rest\alpha]$,
so eventually all $L[x]$-cardinals $<\eta_0$ are also $N$-cardinals,
so $\eta_0$ is a limit cardinal of $N$. 
To define $\widehat{\Fff}$ and $\left<\Fff^P\right>_{P\in\widehat{\Fff}}$
(witnessing $x$-goodness), fix $P_0\in\Fff$ such that $P_0$ is $(\eta_0,\Ddd)$-stable,
and let $\widehat{\Fff}$ be the set of all $P\in\Fff$
such that $P_0\dashrightarrow P$ and $x$
is extender algebra generic over $P$,
and for such $P$, let $\Fff^P$ be the set of all
$Q$ such that $P\dashrightarrow Q$, $Q|\delta^Q\in P|\eta_0$,
and $P|\delta^P$ is extender algebra generic over $Q$.

We claim these objects witness $x$-goodness.
For the choice of $P_0$ ensures that $i_{PQ}(\Fff^P)=\Fff^Q$ for all
$P,Q\in\widehat{\Fff}$ with $P\dashrightarrow Q$.
Clearly $\widehat{\Fff}\sub_{\mathrm{cof}}\Fff$, and $\Fff,\mathscr{F}^P$
are cofinal because  $\Fff^P\sub\Fff$, and the cofinality of $\Fff^P$
in $\Fff$ is by  Boolean-valued
comparison: Let $Q\in\mathscr{F}$
and $\dot{Q}\in P$ be a $\Coll(\om,\beta)$-name
for $Q$ where $\beta<\eta$, such that
$\Coll(\om,\beta)$ forces that $\dot{Q}$ is $M_1$-like.
Working in $P$,
we can form an iteration tree on $P$, ``comparing''
$P$ with all interpretations of names for $\dot{Q}$.
Because $\eta$ is a limit cardinal in $P$,
this produces an iterate $P'$ of $P$ with $\delta^{P'}<\eta$,
and such that it is forced by $\Coll(\om,\beta)$
that $\dot{Q}\dashrightarrow P'$, and hence $Q\dashrightarrow P'$.
We can then iterate $P'\dashrightarrow P''\in\Fff^P$.
We prove a more subtle
variant of this in detail in Claim \ref{clm:F_dense}.

We now observe that $\HOD^{L(\RR^+)}$ has universe that of $M_\infty[*]$.
For $\mathscr{D}$, and hence $\widetilde{\Ddd}$,
are defined in $L(\RR^+)$ without parameters,
so $M_\infty[*]\sub\HOD^{L(\RR^+)}$.
Now let $\xi\in\OR$ and $X\in\pow(\xi)\inter\HOD^{L(\RR^+)}$
and a formula $\varphi$ and $\gamma\in\OR$
such that $\alpha\in X$ iff $L(\RR^+)\sats\varphi(\gamma,\alpha)$.
We want to see $X\in M_\infty[*]$. But given any $P\in\widehat{\Fff}$,
writing $\CC_{\eta_0}=\Coll(\om,{<\eta_0})$
and $\dot{\RR}_{\eta_0}^+$ the natural name for $\RR^+$,
note that
\[ \alpha\in X\iff P\sats\ \forces_{\CC_{\eta_0}} L(\dot{\RR}_{\eta_0}^+)\sats\varphi(\gamma,\alpha). \]
But for cofinally many $P\in\widehat{\Fff}$, $P$ is $((\gamma,\alpha),\mathscr{D})$-stable (and recall $P$ is $(\eta_0,\mathscr{D})$-stable by definition), so
\[ \alpha\in X\iff M_\infty\sats\ \forces_{\CC_{\eta_0^*}}
 L(\dot{\RR}_{\eta_0^*}^+)\sats\varphi(\gamma^*,\alpha^*).\]
Therefore $X\in M_\infty[*]$, as desired.

Clearly $\eta_0$ is the least measurable of $M_\infty$.
We now want to observe that $\delta_\infty=(\eta_0^+)^{L[x]}$.
For $L[x,G_0]$ has a surjection $\eta_0\to\sup i_{M_1M_\infty}``\gamma^{M_1}_s$,
for each $s\in\vec{\mathscr{I}^{M_1}}$,
by restricting the system $\mathscr{D}$ to pairs $(s',P)$
for some appropriate fixed $s'$. Therefore $\delta_\infty\leq(\eta_0^+)^{L[x,G]}=(\eta_0^+)^{L[x]}$.  But we saw above that $\delta_\infty$ is regular in $L[x]$,
and since $\eta_0<\delta_\infty$, therefore $\delta_\infty=(\eta_0^+)^{L[x]}$.

Say that $\eta_0$ is \emph{dl-stable} (with respect to
the current direct limit system) iff
 $i_{M_1P}(\eta_0)=\eta_0$ for all $P\in\mathscr{F}$.
 Suppose $\eta_0$ is dl-stable. Let $j=i_{M_1M_\infty}:M_1\to M_\infty$.
 Recall we defined $M_1^+$ and $j^+:M_1^+\to M_\infty[\Sigma_\infty]$
 earlier.
 So letting $\delta=\delta^{M_1}$,
 we have $M_1|\delta=M_1^+|\delta$ and $V_\delta^{M_1}=V_\delta^{M_1^+}$
 and $\delta$ is Woodin in both.
 Moreover, $M_1^+$ is iterable,
 via the strategy of $M_1$.
 Moreover, if $\Tt$ is a normal tree on $M_1$, of successor length,
 and $b^\Tt$ does not drop, and $\Tt^+$ is the corresponding
 tree on $M_1^+$, then $M^\Tt_\infty=i^{\Tt^+}(M_1)$,
 so $M^\Tt_\infty\sub M^{\Tt^+}_\infty$,
 and $i^\Tt\sub i^{\Tt^+}$,
 and the natural copy map $\pi_\infty:M^\Tt_\infty\to M^{\Tt^+}_\infty$
 is just the identity.
 The iterability of $M_1^+$ was shown in \cite{hod_as_core_model},
 and the extra facts were shown in \cite{Theta_Woodin_in_HOD}
 and \cite{vm1}.

\section{The $\kappa$-mantle for indiscernible $\kappa$}\label{sec:M_1_kappa-mantle}

In this section we analyze the 
following two related $\kappa$-mantles:
\begin{enumerate}[label=--]
	\item the $\kappa$-mantle of $L[x]$, where $\kappa$ is an
	$x$-indscernible, for a real $x$ with $M_1|\delta^{M_1}\in L_\kappa[x]$, and
 \item the $\kappa$-mantle of $M_1$, where $\kappa$ is an
$M_1$-indiscernible,
\end{enumerate}
proving Theorem \ref{tm:unctbl_cof} in the case that $\eta=\kappa_\alpha^x$
is an $L[x]$-indiscernible; the following theorem is by essentially the same proof, which we leave to the reader:

\begin{tm}\label{thm:M_1_kappa-mantle}
	Assume that $M_1^\#$ exists and is $(0,\OR,\OR)$-iterable.
	Let $\kappa=\kappa_\alpha^x$ be an $M_1$-indiscernible. Then $\Mmm_\kappa^{M_1}$
	is a fully iterable strategy mouse which models $\ZFC$.
\end{tm}

We will prove more general theorems later,
but in this special case there is a different, and simpler,
argument, so we present this first.
\begin{proof}[Proof of Theorem \ref{tm:unctbl_cof} when $\eta=\kappa_\alpha^{x}$]
	We start with the case of $\eta=\kappa_0^x$.
We continue with the setup of \S\ref{sec:background}, in the case of
$\mathscr{D}$ consisting of the $M_1$-like
premice $N$ with $N|\delta^N\in L_{\kappa}[x,G]$,
where $\kappa=\kappa_0^x$.

We first show $M_\infty[*]\sub\Mmm_\kappa^{M_1}$,
a fact which is not new. We have $\Ddd,\Fff,\widehat{\Fff}$ as in \S\ref{sec:background}.
Now $\widehat{\Fff}\inter L[x]$
is dense in the ${<\kappa}$-grounds of $L[x]$.
For let $W\sub M$ be a ${<\kappa}$-ground of $L[x]$.
Using a Boolean-valued comparison argument in $W$,
we can compute a non-dropping iterate $N$ of $M_1$ with
$N|\delta^N\in W$ and $\delta^N<\kappa$.
Let $\PP\in V_\kappa^W$ and $g$ be $(W,\PP)$-generic,
with $W[g]=L[x]$.
Let $\gamma<\kappa$ with $\PP\in V_\gamma^W$.
Then working in $W$, iterate $N\dashrightarrow P$ so that it is forced by $\PP$
that $\dot{x}$ is extender algebra generic over $P$, where $\PP$ forces ``$\dot{x}\in\RR$''
and $\dot{x}$ is a name for $x$.
Then note that $x$ is extender algebra generic over $P$, so $P\in\widehat{\Fff}$,
and $P\sub W$, as desired.

Moreover, each $P\in\widehat{\Fff}$ computes $M_\infty[*]$
in the same manner from the parameter $\kappa$. So
\[ 
M_\infty[*]\sub\bigcap_{P\in\widehat{\mathscr{F}}_\kappa}=\Mmm_\kappa^{L[x]},\]
as desired.

We now proceed to the converse, that $\Mmm_\kappa^{L[x]}\sub M_\infty[*]$.

We first show that $\pow({<\OR})\inter\Mmm^{L[x]}_\kappa\sub M_\infty[*]$.
So let $X\sub\alpha\in\OR$ with $X\in\Mmm_\kappa^{L[x]}$.
Let $j:L[x]\to L[x]$ be an embedding with $\crit(j)=\kappa$.
Then $j``\mathscr{I}^x\sub\mathscr{I}^x$.
Now
\[ j(X)\in\Mmm_{j(\kappa)}^{L[x]}\sub\HOD^{L[x,G]}; \]
the ``$\in$'' is by elementarity,
and the ``$\sub$'' is because $\HOD^{L[x,G]}$ is a ground
for $M_1$ via Vopenka, a forcing of size ${<j(\kappa)}$
(one can compute a bound on the size directly, or just
observe that it has size ${<j(\kappa)}$ because 
$j(\kappa)\in\mathscr{I}^x$ and $\HOD^{L[x,G]}$
is defined over $L[x]$ from parameter $\kappa$).

So we can fix a formula $\varphi$ and $\eta\in\OR$ such that for 
$\alpha\in\OR$, we have
\[ \alpha\in j(X)\iff 
L[x]\sats\Coll(\om,{<\kappa})\forces\varphi(\eta,\alpha),\]
so for all $P\in\widehat{\mathscr{F}}$,
\begin{equation}\label{eqn:P_sees_j(X)} \alpha\in j(X)\iff
 P\sats\Coll(\om,{<\kappa})\forces\varphi(\eta,\alpha). \end{equation}

Fix $P\in\widehat{\mathscr{F}}$ which is $\eta$-stable ($P$ is 
also $\kappa$-stable by \S\ref{sec:background}).
Let $Q\in\widehat{\mathscr{F}}$ with $P\dashrightarrow Q$.

\begin{clm}\label{clm:j(X)_fixed} $i_{PQ}(j(X))=j(X)$.\end{clm}
\begin{proof}
Since $i_{PQ}(\kappa,\eta)=(\kappa,\eta)$, this follows
from (\ref{eqn:P_sees_j(X)}) applied to $P$ and $Q$.
\end{proof}

\begin{clm}\label{clm:commutes}
 $j\com i_{PQ}=i_{PQ}\com j$.
\end{clm}
\begin{proof}
We have $\delta^P\leq\delta^Q<\kappa=\crit(j)$.
Also, $P=L[P|\delta^P]$ and 
$Q=L[Q|\delta^Q]$. So
\[ j\rest P:P\to P\text{ and }j\rest Q:Q\to Q
\text{ are elementary.}\]
Now $P=\Hull^P(\delta^P\cup\mathscr{I})$.
So it suffices to see that the claimed commutativity holds for 
all elements of $\delta^P\cup\mathscr{I}$.

Given $\xi<\delta^P$, since 
$\delta^P\leq\delta^Q=i_{PQ}(\delta^P)<\kappa=\crit(j)$, we have
\[ j(i_{PQ}(\xi))=i_{PQ}(\xi)=i_{PQ}(j(\xi)), \]
as desired. Now let $\mu\in\mathscr{I}$.
Since $j``\mathscr{I}\sub\mathscr{I}$
and by \S\ref{sec:background},
$i_{PQ}\rest\mathscr{I}=\id$, so
\[ j(i_{PQ}(\mu))=j(\mu)=i_{PQ}(j(\mu)), \]
completing the proof.
\end{proof}

\begin{clm}$i_{PQ}(X)=X$.
\end{clm}
\begin{proof}
Let $Y=i_{PQ}(X)$. By Claims \ref{clm:j(X)_fixed} and \ref{clm:commutes}, $j(Y)=j(i_{PQ}(X))=i_{PQ}(j(X))=j(X)$,
but $j$ is injective, so $Y=X$ as desired.
\end{proof}

The fact that $X\in M_\infty[*]$ follows from the previous
claim via the following standard calculation.
Let $X^*=i_{P\infty}(X)\in M_\infty$. Then 
$X^*=i_{Q\infty}(X)$
for all $Q\in\widehat{\mathscr{F}}$
with $P\dashrightarrow Q$, since $i_{PQ}(X)=X$.
Let $\alpha\in\OR$. By taking $Q$ as above and also $\alpha$-stable,
it follows that
\[ \alpha\in X\iff Q\sats\text{``}\alpha\in X\text{''}\iff 
M_\infty\sats\text{``}\alpha^*\in X^*. \]
Since $X^*$ and $*\rest\sup(X)$ are both in $M_\infty[*]$,
therefore $X\in M_\infty[*]$.

Now we know $M_\infty[*]\sats\ZFC$, and have shown 
\[ M_\infty[*]\sub\Mmm_\kappa\text{ and }\pow({<\OR})\inter\Mmm_\kappa\sub 
M_\infty[*].\]
It follows that $\Mmm_\kappa\sub 
M_\infty[*]$.
For suppose not, and let $\eta\in\OR$ be largest
such that $V_\eta^{\Mmm_\kappa}=V_\eta^{M_\infty[*]}$.
Therefore $V_\eta^{\Mmm_\kappa}$ is coded by a set $X$ of ordinals
in $M_\infty[*]$. But $M_\infty[*]\sub\Mmm_\kappa$,
so $X\in\Mmm_\kappa$. It follows that every $Y\in V_{\eta+1}^{\Mmm_\kappa}$
is coded by a set $X_Y\in\Mmm_\kappa$ of ordinals.
Hence $X_Y\in M_\infty[*]$. But then 
$V_{\eta+1}^{\Mmm_\kappa}=V_{\eta+1}^{M_\infty[*]}$, a contradiction,
completing the proof for $\eta=\kappa_0^x$.

We now consider the case that $\eta=\kappa_\alpha^x$ with $\alpha>0$,
and we just sketch the differences. Let $j:L[x]\to L[x]$ be elementary
with $j(\kappa_0^x)=\kappa_\alpha^x$. Define $M',\Sigma',*'$
in $L[x]$ from parameter $\eta$ just as we defined $M_\infty,\Sigma_\infty,*$ before from $\kappa_0^x$.
Since we have $j$
and $j(\mathscr{M}^{L[x]}_{\kappa_0^x})=\mathscr{M}^{L[x]}_\eta$,
we get $\mathscr{M}_\eta^{L[x]}=M'[\Sigma']$
has the right first order properties,
and since we have an iterate $N$ of $M_1$
with $N|\delta^N\in L_{\kappa_0^x}[x]$,
therefore $M'$ is an iterate of $M_1$ and $\Sigma'$
a fragment of its normal strategy (the unique one giving
wellfounded branches).

To show that $M'[\Sigma']$ is iterable,
we show that it is an iterate of $M_\infty[\Sigma_\infty]$,
which we know is iterable.
Now $M'$ is the direct limit of all iterates of $M_1$
in $L_\eta[x]$, as the proof of this fact adapts easily.
Note then it is an iterate of $M_\infty$.
And  $\Sigma_\infty$ is (equivalent to) the branch
through the tree $M_\infty\dashrightarrow(M_\infty)^{M_\infty}$,
where $(M_\infty)^{M_\infty}$ is the direct limit of all
iterates of $M_\infty$ in $M_\infty|(\kappa_0^x)^*$,
and $(\kappa_0^x)^*=\kappa_0^{M_\infty}$.
So $j(\Sigma_\infty)$ is the branch
through the tree $M'\dashrightarrow(M_\infty)^{M'}$,
where $(M_\infty)^{M'}$ is the direct limit of all iterates
of $M'$ in $M'|j((\kappa_0^x)^*)$. Therefore
the following claim completes the proof, since the iteration map $M_\infty\to M'$ also
sends $\kappa_0^{M_\infty}$ to $\kappa_0^{M'}$,
and by \S\ref{sec:background},
the iteration map $M_\infty\to M'$ agrees with the iteration map when
extended to base model
$M_\infty[\Sigma_\infty]$.

\begin{clm}\label{clm:kappa_0_pres}$j(\kappa_0^{M_\infty})=\kappa_0^{M'}$.\end{clm}

\begin{proof}[Proof of Claim  \ref{clm:kappa_0_pres}]
	We have $\kappa_0^{M_\infty}=(\kappa_0^x)^*$, with the $*$-map
	associated to $\mathscr{F}^{L[x]}_{<\kappa_0}$.
So letting $*'$ be the $*$-map associated
	to $\mathscr{F}=\mathscr{F}^{L[x]}_{<\eta}$, we have $j(\kappa_0^{M_\infty})=\eta^{*'}$.
	So we just need to see $\eta^{*'}=\kappa_0^{M'}$.
The case that $\alpha$ is a succesor is a simplification of the limit
case, so  suppose $\alpha$ is a limit.

The \emph{eventual ordertype} $\mathrm{eot}(\gamma)$ of an ordinal $\gamma$
is the least ordinal $\tau$ such that for some $\beta<\gamma$,
we have $\gamma=\beta+\tau$. Let $\tau=\mathrm{eot}(\alpha)$.
Fix $\beta<\alpha$ such that $\beta+\tau=\alpha$.
Let $P\in\widehat{\mathscr{F}}$ with
$\kappa_\beta^x<\delta^P$. Then note that $\kappa_\tau^P=\eta$,
and $\eta$ is $(\mathscr{F},P)$-stable. Let $M''$ be the direct limit
of all iterates of $M_1$ in $L_\eta[x]$, under the iteration maps,
and $\sigma':M'\to M''$ be defined as $\sigma$ was
in \S\ref{sec:background}. As in   \S\ref{sec:background},
we have $M'=M''$. Recall that in \S\ref{sec:background}, 
we showed $\sigma=\id$.
However, $\sigma'\neq\id$.

In fact, 
$\sigma'\rest\delta^{M'}=\id$ and
$\sigma'(\kappa_\xi^{M'})=\kappa_{\tau+\xi}^{M'}$ for all $\xi$.
Equivalently,
\begin{equation}\label{eqn:rg(sigma')}
\rg(\sigma')=\Hull^{M'}(i_{PM'}``(\mathscr{I}^P\cut\eta)\cup\delta^{M'}).\end{equation}
To see this, first note that $\mathscr{I}^x\cut\eta=\mathscr{I}^P\cut\eta$
and every $\alpha\in\mathscr{I}^x\cut\eta$ is $(\mathscr{F},P)$-stable.
This gives $\supseteq$ of line (\ref{eqn:rg(sigma')}) like before.
For $\sub$,  let $\alpha\in\OR$.
Let $s\in\vec{\mathscr{I}^x}$
and $t$ be a term such that $\alpha=t^{L[x]}(x,s)$. Let $Q\in\widehat{\mathscr{F}}$
with $P\dashrightarrow Q$  be such that $\max(s\inter\eta)<\delta^Q$.
Then considering the forcing relation with the extender
algebra, there is a term $u$ and $\beta<\delta^Q$
such that $\alpha=u^Q(\beta,s\cut\eta)$. Letting $s'\in\vec{\mathscr{I}^Q}$
with $\max(s)<\max(s')$ and $\beta<\gamma^Q_{s'}$, then $\alpha\in H^Q_{s'}$,
and it follows that $\sigma'(i_{(Qs'),\infty}(\alpha))=i_{QM'}(\alpha)$,
which easily suffices.

In particular, we get $\sigma'(\kappa_0^{M'})=\kappa_\tau^{M'}$.
But for eventually all $Q\in\mathscr{F}$ we have $i_{QM'}(\eta)=\kappa_\tau^{M'}$.
Therefore $\eta^{*'}=\kappa_0^{M'}$, as desired.
\end{proof}

This completes the proof for $\eta=\kappa_\alpha^x$ when $\alpha>0$ is a limit;
note now that the successor case is a simplification (and we get $\sigma'=\id$ in that case).
This completes the proof overall.
\end{proof}

\section{Analysis of local mantles of $L[x]$}

We first give a detailed proof of Theorem \ref{tm:delta-cc_mantle}.
We then just give a sketch of the proof of
Theorem \ref{tm:unctbl_cof}, as it is mostly simpler.

\begin{proof}[Proof of Theorem \ref{tm:delta-cc_mantle}]

Adopt the hypotheses of the theorem. So $\delta$ is some
uncountable regular cardinal of $L[x]$,
$\delta\leq$ the least Mahlo $\theta_0$ of $L[x]$,
and $\delta^{M}\leq\delta$ where $M$ is some non-dropping
iterate of $M_1$ with $M|\delta^{M}\in L[x]$.

Recall from \cite{odle_v2}, that for a premouse $P$ and $\delta\leq\OR^P$,
the \emph{meas-lim extender algebra}
 $\BB_{\measlim}^P$ of $P$ at $\delta$ is the variant of the extender
 algebra at $\delta$, in which we only induce axioms with extenders $E\in\es_+^P$
 such that $\nu_E$ is a limit of measurable cardinals of $P$.

\begin{clmtwo}\label{clm:iterate_height_delta}
	There is an iterate $N$ of $M_1$ with $N|\delta^N\in L[x]$
	and $\delta^N=\delta$.
	\end{clmtwo}
\begin{proof}
	If $\delta$ is a successor cardinal of $x$, then  the usual
	argument via $x$-genericity iteration suffices.
	So suppose $\delta$ is a limit cardinal, hence inaccessible, in $L[x]$.
	
	Suppose first that $\delta<\theta_0$. Work in $L[x]$.
 Let $C\sub\delta$ be a club of singular cardinals,
 and form a meas-lim genericity iteration, folding in short linear iterations
 at successor measurable cardinals, such that every measurable cardinal of the eventual model is in $C$. (That is, the tree $\Tt$ will not drop
 anywhere. Given $\Tt\rest(\alpha+1)$, first let $G^\Tt_\alpha\in\es_+(M^\Tt_\alpha)$ induce the least violation of a meas-lim extender algebra axiom (so $\nu(G^\Tt_\alpha)$ is a limit of measurables of $M^\Tt_\alpha$). If every measurable $\mu$ of $M^\Tt_\alpha$ such that $\mu<\nu(G^\Tt_\alpha)$ is in $C$, then set $E^\Tt_\alpha=G^\Tt_\alpha$,
 and otherwise set $E^\Tt_\alpha=$ the order $0$ measure
 on the least measurable $\mu\notin C$.) Almost the usual proof
 shows that if $\Tt$ reaches length $\delta$, then $\Tt\rest\delta$ is maximal, i.e. $L[M(\Tt)]\sats$``$\delta$ is Woodin'': if not then we get a branch $[0,\delta)_\Tt$, and taking $\pi:M\to V$ an embedding with $\crit(\pi)=\kappa<\delta$ and $\pi(\kappa)=\delta$ and $\Tt,C\in\rg(\pi)$, the usual
 proof shows that the first extender used on $(\kappa,\delta]_\Tt$
 does not violate any extender algebra axioms, and hence $\kappa\notin C$.
Now suppose we reach a maximal tree $\Tt$ but $\delta(\Tt)<\delta$.
Then since $\delta(\Tt)$ is a limit of measurables of $M(\Tt)$,
$\delta(\Tt)\in C$, and hence $\delta(\Tt)$ is singular in $L[x]$.
But $x$ is meas-lim extender algebra generic over $L[M(\Tt)]$,
so $\delta(\Tt)$ is regular in $L[M(\Tt)][x]=L[x]$, a contradiction.
So $\delta(\Tt)=\delta$ as desired.

Now suppose instead that
 $\delta=\theta_0$. Then we proceed much as above,
 but this time we arrange that every measurable of the eventual model
 is regular in $L[x]$.
Suppose that $\Tt\rest\theta_0$ is non-maximal.
Then by Mahloness, we can find an elementary $\pi:M\to V$
with $\crit(\pi)=\kappa$ regular in $L[x]$
and $\pi(\kappa)=\theta_0$. Thus, we reach a contradiction like before.
Now suppose $\Tt$ is maximal with $\delta(\Tt)<\theta_0$.
We claim that $\delta(\Tt)$ is Mahlo in $L[x]$, a contradiction.
For letting $C\sub\delta(\Tt)$ be club and $N=L[M(\Tt)]$,
 since $\BB_{\measlim}^{N}$
is $\delta^N$-cc in $N$, there is a sub-club $D\sub C$ with $D\in N$,
but then we can find some $\kappa<\delta^N$ which is $({<\delta^N},D)$-reflecting
in $N$, which implies $\kappa\in D$, and since $\kappa$ is measurable in 
$N$, it is regular in $L[x]$, which suffices.
	\end{proof}

(For the rest of the proof, we can reduce the assumption
that $\delta\leq\theta_0$, to $\delta<\kappa_0^x$ being regular
in $L[x]$ and having some iterate $N$ of $M_1$ with $\delta^N=\delta$
and $N|\delta\in L[x]$.)\footnote{Thus, one should be able to go
	significantly beyond the least Mahlo.
	However, it is easy to see that we \emph{cannot} get this situation
	with $\delta$ weakly compact in $L[x]$, and in particular,
not at $\delta=\kappa_0^x$, and hence not at measure one many
$\delta<\kappa_0^x$.}

Let $\mathscr{F}=\mathscr{F}^{L[x]}_{\leq\delta}$.
Let $\widehat{\mathscr{F}}$ be the set of all $N\in\mathscr{F}$
such that $x$ is extender algebra
generic over $N$.
Note then that for $N\in\widehat{\mathscr{F}}$,
 we have $N[x]=L[x]$ and by \cite{farah},
$\BB^N$ is $\delta$-cc in $L[x]$, so
 $N\in(\mathscr{G}^{\text{ext}}_{\delta})^{L[x]}$.
 So we have
 \[ \widehat{\mathscr{F}}\sub(\mathscr{G}^{\text{ext}}_\delta)^{L[x]}\sub
  (\mathscr{G}^{\text{int}}_\delta)^{L[x]}.
 \]
 Let $\mathscr{D}$ be the set of
 all $M_1$-like premice $N$ in $L[x]$
 with $\delta^N\leq\delta$. 
 Let $\widehat{\mathscr{D}}$ be the set of those $P\in\mathscr{D}$
 such that $x$ is extender algebra generic over $P$.
 So $\mathscr{F}\sub\mathscr{D}$ and $\widehat{\mathscr{F}}\sub\widehat{\mathscr{D}}$, and $\widehat{\mathscr{D}},\mathscr{D}$ are definable
 from the parameter $\delta$ over $L[x]$.

\begin{clmtwo}\label{clm:F_dense}We have:
\begin{enumerate}
	\item\label{item:immediate}	  $\widehat{\mathscr{F}}$ is directed and dense in $\mathscr{F}$
	(with respect to $\dashrightarrow$), and also dense in $\mathscr{D}$.
	
\item\label{item:main} $\widehat{\mathscr{F}}$ is dense in $(\mathscr{G}^{\text{int}}_{\delta\mathrm{-cc}})^{L[x]}$
 (with respect to $\sub$). Therefore
 \[
\bigcap\mathscr{F}=(\Mmm^{\text{int}}_\delta)^{L[x]}=
(\Mmm^{\text{ext}}_\delta)^{L[x]}.\]
\end{enumerate}
\end{clmtwo}
\begin{proof}
	Part \ref{item:immediate}: Just use standard pseudo-comparison and pseudo-genericity iterations, with
	the regularity of $\delta$.
	
	Part \ref{item:main}:
Let $W\in(\mathscr{G}^{\text{int}}_\delta)^{L[x]}$. 
We must find some $N\in\widehat{\mathscr{F}}$ with $N\sub W$.
Let $M\in\widehat{\mathscr{F}}$ with $\delta^M=\delta$ and $M|\delta^M\in L[x]$.
Fix $\PP\in W$ such that
 $W\sats$``$\PP$ is $\delta$-cc''
and  $W[G]=L[x]$ for some $G$ which is $(W,\PP)$-generic.
Let $\dot{x},\dot{M}\in W$ be $\PP$-names
for $x$ and $M|\delta$ respectively.
We may assume that ($*$) $\PP$ forces ``$L[\dot{M}]$ is an $M_1$-like
premouse which is
 short-tree iterable,
and $\OR(\dot{M})\leq\delta$'', and that either
 $M\notin W$ or $x$ is not
$M$-generic for the extender algebra.

We construct a $\PP$-name $\dot{\Tt}\in W$ and a premouse
$\bar{N}\in W$ such that $N=L[\bar{N}]$ is $M_1$-like
and $\OR(\bar{N})=\delta$
is Woodin in $N$,
and such that $\PP\forces$
``$\dot{\Tt}$ is a padded normal tree on 
$\dot{M}$, and $\lh(\dot{\Tt})=\delta$,
$\bar{N}=M(\dot{\Tt})$,
and $\dot{x}$ is generic over $L[M(\dot{\Tt})]$ for the extender
algebra at $\delta$''.
Note then that $\Tt=\dot{\Tt}_G$ is indeed such a tree on $M_1$,
and $N\in\mathscr{F}$ is as desired.

The construction is a Boolean-valued comparison/genericity iteration.
We construct a sequence
$\left<\eta_\alpha\right>_{\alpha<\lambda}\sub\delta$,
with $\lambda\leq\delta$ (actually $\lambda=\delta$)
and determine $\dot{\Tt}\rest(\alpha+1)$ and 
$N||\eta_\alpha$,
by recursion on $\alpha$,
and for the eventual $\dot{\Tt}$,
we will have  $\PP\forces$``$\dot{\Tt}$ is a padded
pseudo-normal tree on $\dot{M}$, $\lh(\dot{\Tt})=\lambda$, and for each 
$\alpha+1<\lambda$, if $E^{\dot{\Tt}}_\alpha\neq\emptyset$
then $\lh(E^{\dot{\Tt}}_\alpha)=\eta_\alpha$''.

Limit stages are of course dealt with by the fact that
$\PP\forces$``$\dot{M}$ is short-tree-iterable'';
we stop if we reach a limit stage for which there is no Q-structure.
So suppose we have determined $\dot{\Tt}\rest(\alpha+1)$
and $\eta_\beta$ for all $\beta<\alpha$, and 
$N|\sup_{\beta<\alpha}\eta_\beta$,
which is passive. Let $\eta'_\alpha$ be the least $\eta'$
such that $\PP$ does not force a value
for $M^{\dot{\Tt}}_\alpha|\eta'$,
if such exists; otherwise let $\eta'_\alpha=\infty$.
Now if $\eta'_\alpha=\infty$ and $M^{\dot{\Tt}}_\alpha$
is not proper class, i.e. $[0,\dot{\Tt}]_\alpha$
is forced to drop, then we stop the construction.
Now suppose that if $\eta'_\alpha=\infty$ then $M^{\dot{\Tt}}_\alpha$
is proper class.
Let $\eta_\alpha$ be the least $\eta$ such that
either $\eta=\eta'_\alpha$, or $\eta<\eta'$
and $E=F(M^{\dot{\Tt}}_\alpha|\eta)\neq\emptyset$,
$\nu_E$ is a cardinal in 
$M^{\dot{\Tt}}_\alpha||\eta'_\alpha$,
and some $p\in\PP$ forces ``$E$ induces an extender algebra
axiom $\varphi$ and $\dot{x}\sats\neg\varphi$''.
If $\eta_\alpha=\infty$ then we stop the process.
If $\eta_\alpha<\infty$ then we continue, setting
$E^{\dot{\Tt}}_\alpha=F(\dot{M}|\eta_\alpha)$
(which some condition in $\PP$ forces non-empty).
This completes the construction.

\begin{sclm}
The construction proceeds through $\delta$ stages,
 $\sup_{\alpha<\delta}\eta_\alpha=\delta$,
 and $\delta$ is Woodin in $L[N|\delta]$.
\end{sclm}
\begin{proof}
Suppose first we reach $\dot{\Tt}\rest(\alpha+1)$
such that $\PP$ forces that $\eta'_\alpha=\infty$
and that some $p\in\PP$ forces that $[0,\alpha]_{\dot{\Tt}}$
drops. The note that $\PP$ forces a value $Q$ for
$M^{\dot{\Tt}}_\alpha$ (meaning there is $Q\in W$
and $\PP\forces$``$M^{\dot{\Tt}}_\alpha=\check{Q}$'')
and $Q$ is not sound. Thus, $\PP\forces$``$[0,\alpha]_{\dot{\Tt}}$ 
drops and there is 
$\beta<_{\dot{\Tt}}\alpha$
such that $\core_\om(Q)\pins M^{\dot{\Tt}}_\beta$,
and $\beta'\in[\beta,\alpha)$ such that
$E^{\dot{\Tt}}_{\beta'}\neq\emptyset$,
and letting $\beta'$ be least such, then 
$\eta_{\beta'}\leq\OR(\core_\om(Q))$,
and so $E^{\dot{\Tt}}_{\beta'}\in\es_+(\core_\om(Q))$''.
Therefore $\PP$ forces that $E=E^{\dot{\Tt}}_{\beta'}$
is the least difference between $\core_\om(Q)$
and $Q$, and since $\core_\om(Q),Q\in W$ are uniquely determined,
therefore $E\in W$ is also, as is $\beta'$. Since 
$\PP\forces$``$E^{\dot{\Tt}}_{\beta'}=E$'', we have 
$\eta_{\beta'}<\eta'_{\beta'}$,
and since  $\PP\forces$``$\core_\om(Q)\pins M^{\dot{\Tt}}_{\beta'}$'',
we must have $\OR(\core_\om(Q))<\eta'_{\beta'}$
(since if $\eta'_{\beta'}<\infty$ then $M^{\dot{\Tt}}_{\beta'}|\eta'_{\beta'}$
is not uniquely determined by $\PP$).  Because 
$\eta_{\beta'}<\eta'_{\beta'}$,
$\nu_E$ is a cardinal of $M^{\dot{\Tt}}_{\beta'}||\eta'_{\beta'}$.
But then letting $F$ be the extender used in $\dot{\Tt}$
which causes the drop in model to $\core_\om(Q)$,
we have $\crit(F)<\nu_E$ and $F$ is total over $M^{\dot{\Tt}}_{\beta'}|\nu_E$,
and hence total over $M^{\dot{\Tt}}_{\beta'}|\eta'_{\beta'}$,
which contradicts the fact that $\core_\om(Q)\pins 
M^{\dot{\Tt}}_{\beta'}||\eta'_{\beta'}$.

Now suppose we reach $\dot{\Tt}\rest(\alpha+1)$, with $\alpha<\delta$,
such that $\eta'_\alpha=\infty$ but $\PP$ forces that $[0,\alpha]_{\dot{\Tt}}$
does not drop, and that the process stops at this stage.
Then $\PP\forces$``$\dot{x}$ is extender algebra generic
over $M^{\dot{\Tt}}_\alpha$''. But then
letting $G$ be $(W,\PP)$-generic with $W[G]=L[x]$,
we get $M^{\dot{\Tt}_G}_\alpha[x]=L[x]$,
but because $\alpha<\delta$,
we have $i^{\dot{\Tt}_G}_{0\alpha}\in L[x]$,
which implies $M^{\dot{\Tt}_G}_0=M_1$
and the embedding is the identity (see
\cite{gen_kunen_incon}, which shows that a non-trivial embedding
is impossible here). So $M|\delta^{M}\in W$
and $x$ is extender algebra generic over $M$,
contradicting ($*$).

So the process goes through $\delta$ stages,
yielding $\bar{N}=M(\dot{\Tt})\in W$ of height $\delta$.

Now suppose that there is some Q-structure
for $\bar{N}$ of the form $Q=\J_\alpha(\bar{N})$.
Then $\PP$ forces that there is a $\dot{\Tt}$-cofinal
wellfounded branch $\dot{b}$ determined by $Q$. Since 
$\dot{b}=[0,\delta]_{\dot{\Tt}}$
is forced to be club
in $\delta$ and $\PP$ is $\delta$-cc in $W$,
we can fix a club $C\in W$ such that $\PP\forces$``$C\sub 
\dot{b}$ and for all $\kappa\in C$,
we have $\kappa=\crit(i^{\dot{\Tt}}_{\kappa\delta})$
and $i^{\dot{\Tt}}_{\kappa\delta}(\kappa)=\delta$''.

Now given $\kappa\in C$, let $f(\kappa)$ be the least
$\mu\in C$ such that $\kappa<\mu$
and for all $A\in\pow(\kappa)\inter\bar{N}$,
and all $B_0,B_1\in\pow(\delta)$ with $B_0\neq B_1$,
if there are $p_0,p_1\in\PP$ such that
$p_i\forces i^{\dot{\Tt}}_{\kappa\delta}(A)=B_i$
for $i=0,1$,
then $B_0\inter\mu\neq B_1\inter\mu$.
By the $\delta$-cc, $f(\kappa)$ exists.
Let $D$ be the set of limit points of $C$ which are closed
under $f$. For each $\kappa\in D$, there is 
$\pi_{\kappa\delta}:\pow(\kappa)\inter\bar{N}\to\pow(\delta)$ with $\pi_{\kappa\delta}\in W$ 
such that
$\PP\forces$``$i^{\dot{\Tt}}_{\kappa\delta}\rest(\pow(\kappa)\inter\bar{N}
)=\check{\pi}_{\kappa\delta}$''.
For if $A\in\pow(\kappa)\inter\bar{N}$ and
$p_i,B_i$ are as above, then note that there are $p'_i\leq p_i$
and some $\gamma\in\kappa\inter C$
such that $p'_i\forces A=i^{\dot{\Tt}}_{\gamma\kappa}(A\inter\gamma)$
for $i=0,1$, so $p'_i\forces i^{\dot{\Tt}}_{\gamma\delta}(A\inter\gamma)=B_i$,
although $p'_i\forces i^{\dot{\Tt}}_{\gamma\kappa}(A\inter\gamma)=A$,
which contradicts the fact that $\kappa$ is closed under $f$.

Fix $\kappa\in D$. Because $\PP\forces 
i^{\dot{\Tt}}_{\kappa\delta}=\pi_{\kappa\delta}$,
the first extender $E$ (with critical point $\kappa$) forming 
$\pi_{\kappa\delta}$ was chosen
for genericity iteration purposes (and note $E\in W$).
So letting 
$\theta=\nu_E$,
there is
 some $p\in\PP$
and some sequence $\left<\varphi_\alpha\right>_{\alpha<\kappa}\in\bar{N}$,
with $\varphi_\alpha\in \bar{N}|\kappa$ for each $\alpha<\kappa$,
such that 
\[ p\forces\text{``}x\sats\neg\bigvee_{\alpha<\kappa}\varphi_\alpha\text{ but }
 x\sats\bigvee_{\alpha<\theta}\varphi_\alpha\text{''},
\]
where 
$\left<\varphi_\alpha\right>_{\alpha<\theta}=i_E(\left<\varphi_\alpha\right>_{
\alpha<\kappa})\rest\theta$, and $\varphi_\alpha\in\bar{N}|\theta$
for each $\alpha<\theta$.
Letting
\[ \left<\varphi_\alpha\right>_{\alpha<\delta}=\pi_{\kappa\delta}
(\left<\varphi_\alpha\right>_{\alpha<\kappa}),\]
we therefore have
 \[ p\forces\text{``}\dot{x}\sats\neg\bigvee_{\alpha<\kappa}\varphi_\alpha
\text{ but }\dot{x}\sats\bigvee_{\alpha<\delta}\varphi_\alpha\text{''}.\]
Let $g(\kappa)$ be the least $\mu\in D$
such that for each sequence 
$\left<\psi_\alpha\right>_{\alpha<\kappa}\in\bar{N}$ with
$\psi_\alpha\in\bar{N}|\kappa$ for each $\alpha<\kappa$,
defining $\left<\psi_\alpha\right>_{\alpha<\delta}$ as above, for each 
$p\in\PP$,  if
\[ 
p\forces\text{``}\dot{x}\sats\neg\bigvee_{\alpha<\kappa}\psi_\alpha\text{ 
but }\dot{x}\sats\bigvee_{\alpha<\delta}\psi_\alpha\text{''},\]
then
\[ p\forces\text{``}\dot{x}\sats\bigvee_{\alpha<\mu}\psi_\alpha\text{''};\]
again this exists by the $\delta$-cc,
and since there are ${<\delta}$-many sequences to consider.

Now let $\kappa$ be a limit point of $D$ which is closed under $g$.
Let $\left<\varphi_\alpha\right>_{\alpha<\delta}$ and $p$ be as before
(which exist as mentioned above).
Let $p'\leq p$ and $\gamma\in D\inter\kappa$
be such that 
$p'\forces$``$i^{\dot{\Tt}}_{\gamma\kappa}(\left<\varphi_\alpha\right>_{
\alpha<\gamma})=\left<\varphi_\alpha\right>_{\alpha<\kappa}$''.
Then 
$p'\forces$``$i^{\dot{\Tt}}_{\gamma\delta}
(\left<\varphi_\alpha\right>_{\alpha<\gamma})=
\left<\varphi_\alpha\right>_{\alpha<\delta}$'',
and since $p'\leq 
p\forces$``$\dot{x}\sats\bigvee_{\alpha<\delta}\varphi_\alpha$''
and $\kappa$ is closed under $g$,
it follows that 
$p'\forces$``$\dot{x}\sats\bigvee_{\alpha<\kappa}\varphi_\alpha$'',
contradicting that $p$ forces the opposite.

So there is no Q-structure as above,
so $L[\bar{N}]\sats$``$\delta$ is Woodin'',
completing the proof of the subclaim.
\end{proof}

Because we weaved in genericity iteration, $\PP\forces$``$\dot{x}$
is extender algebra generic over $L[\bar{N}]$'',
and therefore $x$ is extender algebra generic over $L[\bar{N}]$.
But $\bar{N}\in W\sub L[x]$,
so $L[\bar{N}]$ is a ground of $L[x]$ and $L[\bar{N}]\sub W$.
As $L[\bar{N}]\in\mathscr{F}$, we are done.
\end{proof}

\begin{clmtwo}
For each $W\in\mathscr{G}^{\mathrm{int}}_{\delta\mathrm{-cc}}$, we have:
 \begin{enumerate}[label=--]
 	\item  $W=L[A]$ for some $A\sub\delta$,
 	\item  $W$ is a ground of $L[x]$ via a forcing
 $\QQ\sub\delta$ such that 
$W\sats$``$\QQ$ is $\delta$-cc''.
\end{enumerate}
\end{clmtwo}

 \begin{sclm}\label{clm:W=L[B]}
  There is $\eta\in\OR$ and $B\sub\eta$ such that $W=L[B]$.
 \end{sclm}
\begin{proof}
 Let $(\PP,G)$ witness that $W$ is a ground of $L[x]$
 and $\eta\in\OR$ with $\PP\sub\eta$. Let $\tau\in W$
 be a $\PP$-name for $x$. Let $W'=L[\PP,\tau]$. So $W'\sub W$.
Note that $G$ is $(W',\PP)$-generic and $W'[G]=L[x]$.
But since $W'\sub W$ and $W[G]=L[x]$, a standard forcing computation
shows that $W'=W$.
\end{proof}

Fix $B,\eta$ as in the subclaim.
Fix $N\in\mathscr{F}$ witnessing Claim \ref{clm:F_dense}.
Let $\BB=\BB^N_\delta$.
Let $\dot{B}\in N$ be a 
$\BB$-name
for $B$.
For $\alpha<\eta$ let $v_\alpha\in\BB$
be $v_\alpha=||\text{``}\check{\alpha}\in\dot{B}\text{''}||_\BB$
(where $||\varphi||_\BB$ is the Boolean value of $\varphi$ in $\BB$).
Working in $N$, let $\CC$ be the complete Boolean
subalgebra of $\BB$ generated by
$\{v_\alpha\}_{\alpha<\eta}$.
Since $\BB$ is itself a complete Boolean algebra in $N$,
we have $\CC\sub\BB$ and $\CC$ is a regular subalgebra of $\BB$.
Let $H=G\inter\CC$. Then $H$ is $(N,\CC)$-generic
and note that $N[H]=N[B]=L[B]=W$. Since $N=L[N|\delta]$,
we get $W=L[N|\delta,H]$, and $N|\delta,H\sub\delta$.
So we have established the first part of the theorem,
i.e. regarding $W=L[A]$.

In $W$, let $\QQ=\BB/H$. Then because $\BB$ is a $\delta$-cc complete
Boolean algebra in $N$, $\QQ$ is $\delta$-cc in $W$,
$\QQ\sub\delta$, and $W$ is a ground of $W$ via $\QQ$.

This is a standard construction, but here is a reminder:
For $p\in\BB$ say $p\compat H$ iff $p\compat q$
for every $q\in H$, and $p\incompat H$ otherwise.
Define
\[ p\leq_H q\iff (p\wedge(\neg q))\incompat H, \]
(which is a quasi-order on $\BB$)
and define an equivalence relation on $\BB$ by
\[ p\approx_H q\iff p\leq_H q\leq_H p,\]
with equivalence classes $[p]_H$.
Then $\QQ$ be the partial order whose conditions are the equivalence 
classes $[p]_H$ excluding $[0]_H$
(that is, $[p]_H\in\QQ$ iff $p\compat H$),
and ordering induced by $\leq_H$.
Note that given $p,p'\in\BB$ with $[p]_H,[p']_{H}\neq 0$,
we have $[p]_H\compat[p']_H$ iff $p\cdot p'\compat H$. Standard
texts contain further details (***find?).

Write now $\Mmm=\bigcap\widehat{\mathscr{F}}$.
We want to analyze $\Mmm$. For $P\in\widehat{\mathscr{F}}$,
let $\mathscr{F}^P$ be the set of all iterates $N$ of $P$
with $N\in P$ and $\delta^N=\delta^P$ and $P|\delta^P$
being $N$-generic for $\BB^N$. Then:
\begin{clmtwo}
	$\mathscr{D}$ is $x$-good, as witnessed by $\mathscr{F},\widehat{\mathscr{F}},\eta=G=\emptyset,\left<\mathscr{F}^P\right>_{P\in\widehat{\mathscr{F}}}$.
\end{clmtwo}
\begin{proof}
	To see that $\mathscr{F}^P$ and $\mathscr{F}$ are cofinal,
	note that $\mathscr{F}^P\sub\mathscr{F}$,
	and given $N\in\mathscr{F}$, since $P$ is a  $\delta$-cc
	ground of $L[x]$, the proof of Claim \ref{clm:F_dense}
	shows there is $Q\in\mathscr{F}^P$ with $N\dashrightarrow Q$.
	\end{proof}

So \S\ref{sec:background} applies (excluding the last
part which dealt with the case of a limit cardinal $\eta_0$). Thus,
we have $M_\infty$ and $M_\infty[*]=M_\infty[\Sigma_\infty]$.

\begin{clmtwo}
 $M_\infty[*]\sub\Mmm$.
\end{clmtwo}
\begin{proof}
 Let $N\in\mathscr{F}$.
By \S\ref{sec:background},
working in $N$, we can
compute the $(M_\infty)^{\mathscr{F}^N}$ and $*^{\mathscr{F}^N}$
	from $\mathscr{F}^N$. But these are just $M_\infty,*$.
\end{proof}

The following claim is the key fact:
\begin{clmtwo} $\Mmm\sub M_\infty[*]$.\end{clmtwo}\label{clm:Mmm_sub_M_infty*}
\begin{proof}\setcounter{sclm}{0}
It suffices here to see that
\begin{equation}\label{eqn:pow(OR)_sub}\pow(\OR)\inter 
\Mmm\sub M_\infty[*].\end{equation}
For given this, a standard argument shows that for
 each $\alpha\in\OR$, we have 
\[ V_\alpha\inter\Mmm=V_\alpha\inter M_\infty[*], \]
and that $V_\alpha\inter\Mmm\in\Mmm$.
For given this at $\alpha$,
then since $M_\infty[*]\sats\ZFC$,
we get some $\beta\in\OR$ and $X\sub\beta$
such that $X$ codes $V_\alpha\inter\Mmm$,
and since $M_\infty[*]\sub\Mmm$,
therefore $X\in\Mmm$. So if $A\sub V_\alpha\inter\Mmm$
and $A\in\Mmm$, then since each $N\in\mathscr{F}$
models $\ZFC$, there is $A'\sub\beta$ coding
$A$ with respect to $(\beta,X)$
and such that $A'\in\Mmm$, and so by line (\ref{eqn:pow(OR)_sub}),
we have $A'\in M_\infty[*]$.
Therefore $V_{\alpha+1}\inter\Mmm= V_{\alpha+1}\inter M_\infty[*]$.
But also $V_{\alpha+1}\inter M_\infty[*]\in\Mmm$,
completing the induction step. Limit stages now follow easily.

So let $X\in\pow(\OR)\inter\Mmm$
and suppose $X\notin M_\infty[*]$.
We will reach a contradiction.

\begin{sclmtwo}\label{sclm:eventual_move_X}
 For every $N\in\mathscr{F}$ there
 is $P\in\mathscr{F}$ with $N\dashrightarrow P$ and such that
 for every $Q\in\mathscr{F}$
 with $P\dashrightarrow Q$ we have $i_{NQ}(X)\neq X$.
\end{sclmtwo}
\begin{proof}
 Otherwise we can fix a dense $\mathscr{F}'\sub\mathscr{F}$
 such that $i_{NQ}(X)=X$ for all $N,Q\in\mathscr{F}'$
 with $N\dashrightarrow Q$. 
 But then letting $X^*=i_{NM_\infty}(X)$
 for any $N\in\mathscr{F}'$,
 we have $X^*\in M_\infty$,
 and note that $X\in M_\infty[*]$,
 because $\alpha\in X$ iff $\alpha^*\in X^*$, for each $\alpha\in\OR$.
 This is a contradiction.
\end{proof}

Now let $s\in[\mathscr{I}]^{<\om}\cut\{\emptyset\}$
and $t$ be a term such that $X=t^{L[x]|\max(s)}(s^-,x)$.

\begin{sclmtwo}\label{sclm:X_in_H^N_s}
 For every $N\in\widehat{\mathscr{F}}$ we have $X\in H^N_s$.
\end{sclmtwo}
\begin{proof}
Let $\mathscr{X}^N$
denote the set of all $Y$ such that for some $p\in\BB^N$,
we have $p\forces$``$V=L[\dot{x}]$ and 
$Y=t^{L[\dot{x}]|\max(s)}(s^-,\dot{x})$''.
For each such $Y$, let $p_Y$ be the value of this statement
in $\BB^N$. 
So
$\{p_Y\bigm|Y\in\mathscr{X}^N\}$
is an antichain,
and note that it is definable over $N|\max(s)$ from the parameter
$s^-$, so is in $H^N_s$, and the ordertype of its natural enumeration is 
$<\gamma^N_s$. It follows that $\mathscr{X}^N\sub H^N_s$,
but $X\in\mathscr{X}^N$, so we are done.
\end{proof}

Fix $N\in\widehat{\mathscr{F}}$.
Let $\widetilde{\mathscr{F}}^N$ be the set of all $P\in\mathscr{F}^N$
such that $N\sats$``$\BB^N$ forces that
  $\dot{x}$ is extender algebra generic over $P$''. Note that
  $\widetilde{\mathscr{F}}^N$ is dense in $\mathscr{F}^N$,
  by simultaneously iterating to make $\es^N$ generic
  and to force $\dot{x}$ to be generic.
   Let  $\mathscr{Y}^N$
be the set of all $Y\in\mathscr{X}^N$
such that $Y$ is a set of ordinals and 
there is some $P\in\widetilde{\mathscr{F}}^N$ such that for all 
$Q\in\widetilde{\mathscr{F}}^N$ with
$P\dashrightarrow Q$, we have $i^{NQ}_{ss}(Y)\neq Y$. Let $P^N_Y$
be the witness $P$ to which is least in the $N$-order.
Let then $\Tt^N$ be the tree on $N$ given by comparing
all $P^N_Y$ for $Y\in\mathscr{Y}^N$ and simultaneously iterating
to make $\es^N$ generic and to force (with $\BB^N$) $\dot{x}$ to be generic.
Because $\mathscr{Y}^N$
has cardinality ${<\delta^N}$ in $N$,
we get $M^{\Tt^N}_\infty\in\widetilde{\mathscr{F}}^N$. (More formally,
$\Tt^N$ is a maximal limit length tree and $L[M(\Tt^N)]\in\widetilde{\mathscr{F}^N}$.)
And note that we have defined $\Tt^N$ over $N$ from $s$,
uniformly in $N\in\widehat{\mathscr{F}}$.

Now let $N_0=N$ and given $N_n$, let $N_{n+1}=M^{\Tt_{N_n}}_\infty$.
So note each $N_{n+1}\in\widehat{\mathscr{F}}$, and $N_n\dashrightarrow N_{n+1}$,
and $N_{n+1}\in\widetilde{\mathscr{F}}^{N_n}$,
and $\Tt^{N_{n+1}}=j_{n,n+1}(\Tt^{N_n})$
where $j_{n,n+1}=i_{N_n,N_{n+1}}$.
Moreover, $j_{n,n+1}(X)\neq X$.
Let $N_\om$ be the direct limit of the $N_n$, for $n<\om$.
Let $j_{n\om}:N_n\to N_\om$ be the iteration map.

Because of the uniform definability (in particular
that $\Tt^{N_n}$ is defined over $N_n$ from $s$, uniformly in $n$,
as $j_{n,n+1}(s)=s$ for all $n<\om$),
the material in \S\ref{sec:background} adapts routinely to the direct limit
system $\mathscr{F}'=\{N_n\}_{n<\om}$. 
Let $*'$ be the associated map;
so $\alpha^{*'}=\lim_{n\to\om}j_{n\om}(\alpha)$.
We have that $N_\om$ and $*'$ are uniformly definable
over $N_n$ from $s$, and noting that $j_{0\om}(\delta)=\delta$, we have $V_{\delta}^{N_\om[*']}=V_\delta^{N_\om}$,
$\delta$ is Woodin in $N_\om[*']$, $*'$
is the restriction of the iteration map $N_\om\to N_{\om,\om+\om}$,
where the $N_{\om+\beta}$ for $\beta\in[1,\om)$ are produced by continuing the 
process
above (starting with $\Tt_\om=j_{0\om}(\Tt_0)$),
and $N_{\om+\om}$ is the resulting direct limit.
Also by the usual arguments,
\[ \Hull^{N_\om[*']}(\rg(j_{n\om}))\inter N_\om=\rg(j_{n\om}),\]
so letting $N_n^+$ be the transitive collapse of the hull above,
we get $N_n\sub N_n^+$, $V_\delta^{N_n}=V_\delta^{N_n^+}$,
$\delta$ Woodin in $N_n^+$, and letting $j_{n\om}^+:N_n^+\to N_\om[*']$
the uncollapse, we have $j_{n\om}\sub j_{n\om}^+$.

\begin{sclmtwo}$N_\om[*']\in\mathscr{G}^{\mathrm{int}}_{\delta\mathrm{-cc}}$, so $X\in N_\om[*']$.
\end{sclmtwo}
\begin{proof}
Let $\CC'$ be the version of Schindler's forcing (see \S\ref{sec:background})
induced by the system $\{N_n\}_{n<\om}$, recalling here
that $x$ is extender algebra generic over each $N_n$.
As in \S\ref{sec:background},
we have $\CC\sub\delta$, $\CC$ is
$\delta$-cc in $N_\om[*']$, and $G'_x$ is $(N_\om[*'],\PP)$-generic,
and clearly then $N_\om[*'][G'_x]=N_\om[*'][x]=L[x]$.
\end{proof}

\begin{sclmtwo} We have:
	\begin{enumerate}
\item\label{item:first} $N_\om[*']=\dirlim_{n<\om}N_n^+,j_{nm}^+$, and
\item\label{item:second} $j_{n\om}^+$ is a class of $N_n^+$, and $j_{n,n+1}(j_{n\om}^+)=j_{n+1,\om}^+$.
\end{enumerate}
\end{sclmtwo}
\begin{proof}
Part \ref{item:first}: We have already established this below $\delta$
 and for the ordinals of $N_\om[*']$. But since $*'$
 is just the restriction of the iteration map
 for an iteration determined in a manner first-order over $N_\om$
 from $s$, and by the uniqueness of the wellfounded branch,
 it follows that the direct limit above is just $N_\om[*']$.
 
 Part \ref{item:second}:  $j_{n\om}\rest\OR$
 is the transitive collapse of $*'$ to $N_n$,
 and
 $j_{n\om}^+$ is just the ultrapower map
 associated to $\Ult(N_n,E)$, where $E$
 is the $(\delta,\delta)$-extender derived from $j_{n\om}$.
\end{proof}

Now $X\in N_\om[*']$, and by the previous claim,
we can fix $n<\om$ and $\bar{X}_n\in N_n^+$
such that $j_{n\om}^+(\bar{X}_n)=X$.
So $j^+_{n,n+1}(\bar{X}_n)=\bar{X}_{n+1}$
where $j_{n+1,\om}^+(\bar{X}_{n+1})=X$.
But $j_{n,\om}^+$ is a class of $N_n^+$ ,
and $j_{n,n+1}^+(j_{n,\om}^+)=j_{n+1,\om}^+$,
and $j_{n,n+1}^+(N_\om[*'])=N_\om[*']$.
Therefore $j_{n,n+1}^+(X)=X$. But $X\in N_n$
and $j_{n,n+1}\sub j_{n,n+1}^+$,
so $j_{n,n+1}(X)=X$, a contradiction.
\end{proof}

Now by \S\ref{sec:background}, we have:

\begin{clmtwo}
 We have:
 \begin{enumerate}
  \item $V_{\delta_\infty}^{M_\infty[*]}=V_{\delta_\infty}^{M_\infty}$,
  \item $\delta_\infty$ is Woodin in $M_\infty[\Sigma]$,
  \item $\delta$ is the least measurable cardinal of $M_\infty[\Sigma]$,
  \item $M_\infty[\Sigma]$ is ground of $L[x]$
  via a forcing $\CC\sub\delta_\infty$ which is $\delta_\infty$-cc
  in $M_\infty[\Sigma]$.
 \end{enumerate}
\end{clmtwo}

By the preceding claim, $\delta_\infty$ is a regular cardinal of $L[x]$
and $\delta_\infty>\delta$.
\begin{clmtwo}
 $\delta_\infty\leq(\delta^{++})^{L[x]}$.
\end{clmtwo}
\begin{proof}
Equivalently,  for each $s\in\mathscr{I}^{<\om}$,
 we have
 $\gamma^{M_\infty}_s<(\delta^{++})^{L[x]}$. But this is just by cardinality considerations, since
  $L[x]\sats$``$\card(\mathscr{F})\leq\delta^+$''.
\end{proof}

By the previous two claims,
either $\delta_\infty=(\delta^+)^{L[x]}$
or $\delta_\infty=(\delta^{++})^{L[x]}$. So Claim \ref{clm:delta_infty>delta+}
below completes the proof of the theorem:

\begin{rem}
Assume constructible Turing determinacy (see Definition \ref{dfn:con_Turing}).

Using genericity iterations, one easily establishes the following.
Let 
$\eta\in\OR$. Then there is a cone of reals $x$
such that $L[x]\sats\varphi(\eta)$
iff there is a non-dropping iterate $N$ of $M_1$
with $\delta^N<\om_1$ such that for all non-dropping iterates
$P$ of $N$ with $\delta^P<\om_1$, we have $i_{NP}(\eta)=\eta$,
and $N\sats\Coll(\om,\delta^N)\forces\varphi(\eta)$.

Note that
there is also a cone of reals $z$ such that there is $(N,G)$
such that $N$ is a non-dropping iterate of $M_1$
and $G$ is $(N,\Coll(\om,\delta^N))$-generic
and $L[z]=N[G]$. (The existence of $(N,G)$ is just an assertion
in $L[z]$ about a real parameter coding $M_1|\delta^{M_1}$,
and using genericity iterations, one can easily see that
there are at least cofinally many such reals $z$.)
Let us say that $z$ is \emph{sufficient} if $z$ is in this cone.
\end{rem}

\begin{clmtwo}[Steel]\label{clm:delta_infty>delta+}
Assuming constructible Turing determinacy,
$\delta_\infty>(\delta^{+})^{L[x]}$.
\end{clmtwo}
\begin{proof}
For the moment we assume that $\delta=\om_1^{L[x]}$
and $x$ is sufficient, as witnessed by $N,G$. So $L[x]=N[G]$.

Now suppose the claim fails in this case, so $\delta_\infty=\om_2^{L[x]}$.
Because $M_\infty[*]$ is definable without parameters
over $L[x]$, the fact that $\delta_\infty^{L[x]}=\om_2^{L[x]}$
is just an assertion in the langauage of set theory without parameters
over $L[x]$, and is therefore forced by $\Coll(\om,\delta^N)$
over $N$.
So letting $\xi=((\delta^N)^{++})^N=\om_2^{L[x]}$,
we have
\[ 
N\sats\text{``}\Coll(\om,\delta^N)\forces\delta_\infty=\check{\xi}\text{''}.\]
For $n<\om$ let $s_n=(\aleph_0,\ldots,\aleph_n)$.
Let $n<\om$ be least such that 
$(\undertilde{\delta}^1_2)^{L[x]}<\gamma^{M_\infty}_{s_n}$;
this exists since
\[ 
(\undertilde{\delta}^1_2)^{L[x]}<\om_2^{L[x]}=\delta_\infty=
\sup_{n<\om}\gamma^{M_\infty}_{s_n}.\]

Work in $L[x]$.
 Let $W\sub\om_1$ code a wellorder of 
$\om_1$
in ordertype $\gamma^{M_\infty}_{s_{n+1}}$.
Let $\PP$ be almost disjoint forcing for coding $W$ with a real,
with respect to the standard enumeration of $\RR^{L[x]}$.
Recall that $\PP$ is $\sigma$-centered
and hence ccc.
Let $H$ be $(L[x],\PP)$-generic, and $y$ the generic real,
so $L[x,H]=L[x,y]$. Note then that $W$ is $\Delta^1_2(\{(x,y)\})$ in $L[x,y]$
(in the codes given by the $L[x]$-constructibility order of $\RR^{L[x]}$),
so 
\[ (\gamma^{M_\infty}_{s_{n+1}})^{L[x]}<(\undertilde{\delta}^1_2)^{L[x,y]}.\]

But $M_\infty^{L[x]}=M_\infty^{L[x,y]}$,
because (i) $\om_1^{L[x]}=\delta=\om_1^{L[x,y]}$
and (ii) $\mathscr{F}^{L[x]}$ is dense in $\mathscr{F}^{L[x,y]}$,
by  Claim \ref{clm:F_dense} and because $L[x]\sats$``$\PP$ is 
$\delta$-cc''.
Therefore
\begin{equation}\label{eqn:gamma_n+1<bfdelta} 
(\gamma^{M_\infty}_{s_{n+1}})^{L[x,y]}<(\undertilde{\delta}^1_2)^{L[x,y]}.
\end{equation}
But $x\leq_T(x,y)$, so $(x,y)$ is also sufficient. Let $N',G'$ witness this,
so $L[x,y]=N'[G']$.
Because
\[ 
N\sats\Coll(\om,\delta^N)\forces\undertilde{\delta}^1_2<\gamma^{M_\infty}_{s_n},
\]
and $i_{M_1N}(s_n)=s_n=i_{M_1N'}(s_n)$, we have
\[
N'\sats\Coll(\om,\delta^{N'})\forces
\undertilde{\delta}^1_2<\gamma^{M_\infty}_{s_n}.
\]
This contradicts line (\ref{eqn:gamma_n+1<bfdelta}).

Now consider the general case. By constructible Turing determinacy,
we can fix a sufficient real $x'$. So the preceding argument
applies to $L[x']$ and $\delta'=\om_1^{L[x]}$. But for each
$N\in\widehat{\mathscr{F}}'$ (i.e. with $\delta^N\leq\delta'$ and $x'$ being extender algebra
generic over $N$, and hence $\delta^N=\delta'=\om_1^{L[x']}$), we have that $M_\infty^{L[x']}=M_\infty^{\mathscr{F}^N}$ (where $\mathscr{F}^N$ is defined
inside $N$ as before), but also $(\delta^+)^{N}=\omega_2^{L[x]}$ and $(\delta^{++})^{N}=\omega_3^{L[x]}$.
So $N\sats$``$(\delta^N)^{++}$ is the Woodin of $M_\infty^{\mathscr{F}^N}$'',
and this likewise implies the claim.
\end{proof}
This completes the proof of Theorem  \ref{tm:delta-cc_mantle}.\end{proof}

In the case $\delta=\om_1^{L[x]}$ (and constructible Turing determinacy), since $\delta$ is the least measurable
and $\om_3^{L[x]}$ the Woodin, it is natural to ask:
\begin{ques}What is $\om_2^{L[x]}$ in $M_\infty$?
\end{ques}
It is not the least ${<\delta_\infty}$-strong cardinal $\kappa$ of $M_\infty$,
because $\kappa<\om_2^{L[x]}$ (see \S\ref{sec:strong_cardinal}).
Of course Woodin showed that $\omega_2^{L[x]}$ \emph{is}
Woodin in $\HOD^{L[x]}$, and the real mystery and long-standing open question is:
\begin{ques}
	What is $\HOD^{L[x]}$?
	\end{ques}

The analysis of $\mathscr{M}^{L[x]}_\eta$ for $\eta$
such that $L[x]\sats$``$\eta$ is a limit cardinal of uncountable cofinality''
is a simplification of the foregoing argument:

\begin{proof}[Proof Sketch for Theorem \ref{tm:unctbl_cof}]
	Suppose first that $\eta\leq\kappa_0^x$.
	We already dealt with some further aspects of this case in \S\ref{sec:background} (in particular
	that $M_\infty[*]$ is an internally $\delta_\infty$-cc ground of $L[x]$
	and $\delta_\infty=(\eta^+)^{L[x]}$  etc).
The complications to do with getting and keeping all iterates
	of height $\delta$ in the previous proof are  avoided, as standard
	pseudo-comparisons, pseudo-genericity iterations and their variants
	terminate by the next successor cardinal above everything in question.
	However, the proof of Claim \ref{clm:Mmm_sub_M_infty*} does not
	quite go through as before. Subclaims \ref{sclm:eventual_move_X} and \ref{sclm:X_in_H^N_s} 
	are still fine. If $\eta$ is inaccessible in $L[x]$, then the rest adapts
	easily. But if $\eta$ is singular, and we define $\Tt^N$ as in the paragraph
	after Subclaim \ref{sclm:X_in_H^N_s}, then it seems that $\Tt^N$ might
	have length ${\geq\eta}$, and hence lead out of $\mathscr{F}$. 
	So instead of this, working in $L[x]$, let $\Gamma$ be the set of
	all limit cardinals of $L[x]$ which are ${<\eta}$,
	and for each $\gamma\in\Gamma$,
	let $N_\gamma$ be the result of simultaenously comparing all iterates of $M_1$ in $L_\gamma[x]$ and iterating to make $x$ extender algebra generic.
	Then $\delta^{N_\gamma}=(\gamma^+)^{L[x]}$
	and $(\gamma_0<\gamma_1)\Rightarrow (N_{\gamma_0}\dashrightarrow N_{\gamma_1})$, and $\left<N_\gamma\right>_{\gamma\in\Gamma}$ is cofinal
	in $\mathscr{F}_{<\eta}$.  Moreover, $N_\gamma$ is a ${<\eta}$-ground of $L[x]$, so $X\in N_\gamma$, and in fact, as in  Subclaim \ref{sclm:X_in_H^N_s}, we can fix $s\in\vec{\mathscr{I}^x}$ such that
	$X\in H^{N_\gamma}_s$ for each $\gamma$. So by Subclaim \ref{sclm:eventual_move_X}, we can (in $L[x]$) define
	$f:\Gamma\to\Gamma$ by setting $f(\gamma)=$ the least $\gamma'$
	such that
	\[ i_{N_{\gamma}N_{\gamma'}}(X)=i^{N_{\gamma}N_{\gamma'}}_{ss}(X)\neq X.\]
    Now let $\xi<\eta$ be a limit point of $\Gamma$ such that $f``(\Gamma\inter\xi)\sub\xi$. Let $M^\xi_\infty[\Sigma^\xi_\infty]$ be the strategy
    mouse given by the direct limit of $\mathscr{F}_{<\xi}$.
    Then $M^\xi_\infty[\Sigma^\xi_\infty]$ is a ${<\eta}$-ground,
    so $X\in M^\xi_\infty[\Sigma^\xi_\infty]$.
    We can now proceed basically as before for a contradiction.
    (Fix $\gamma_0\in\Gamma\inter\xi$ such that $\xi$
    is $(N_{\gamma_0},\mathscr{F}_{<\xi})$-stable. Then the direct limit system leading to $M^\xi_\infty[\Sigma^\xi_\infty]$
    can be defined over $N$, from parameter $\xi$, uniformly
    in $N\in\mathscr{F}_{<\xi}$ such that $N_{\gamma_0}\dashrightarrow N$. So we can proceed much like before.)
	We
	leave the remaining details to the reader.
	
	Now if $\kappa_\alpha^x<\eta\leq\kappa_{\alpha+1}^x$,
	then it is essentially the same, with $\kappa_{\alpha+1}^x$ replacing $\kappa_0^x$. If $\eta=\kappa_\alpha^x$ for a limit $\alpha$,
	then appeal to \S\ref{sec:M_1_kappa-mantle}
	or argue as there.
\end{proof}

\begin{rem}
	It is of course natural to ask what happens when
	$\eta$ is a limit cardinal of $L[x]$ with $\cof^{L[x]}(\eta)=\om$.
	In this case, we still have $M_\infty[*]\sub \mathscr{M}_\eta^{L[x]}$
	 as before. But it can be that
	$M_\infty[*]\psub\mathscr{M}_\infty^{L[x]}$. 
	For  the most obvious example, consider $\eta=\aleph_\om^{L[x]}$.
	Then note that for any $<\eta$-ground $W$ of $L[x]$,
	$W$ and $L[x]$ agree on a tail of cardinals below $\aleph_\om^{L[x]}$,
	and therefore  $\left<\aleph_n^{L[x]}\right>_{n<\om}\in\mathscr{M}_\eta^{L[x]}$.
	This sequence is, moreover, Prikry generic over $M_\infty$,
	hence over $M_\infty[*]$. (This should be compared to the situation with $L[U]$ and the intersection of its
	finite iterates being  a Prikry extension of its $\om$th iterate.)
	The author does not know whether $\mathscr{M}_\eta^{L[x]}\sats\ZFC$
	when $\eta$ is a limit cardinal  of $L[x]$ with $\cof^{L[x]}(\eta)=\om$.
	\end{rem}

\section{Strategy extensions of $M_1$}\label{sec:strategy_extensions}

Recall the definition of a \emph{nice} extension $M_1[\Sigma]$ of $M_1$ (Definition \ref{dfn:nice_extension}). 
We saw in \S\ref{sec:background} Woodin's Theorem \ref{tm:Woodin_nice_extension}, that if $\kappa$
is the least $M_1$-indiscernible, then $M_1[\Sigma_{\kappa}]$
is nice (recall $\Sigma_{\kappa}=\Sigma_{M_1}\rest(M_1|\kappa)$).\footnote{Moreover,
we have the further property mentioned at the end of \S\ref{sec:background},
that non-dropping iteration maps on $M_1$ extend to those on $M_1[\Sigma_\kappa]$.}
Woodin also showed that, however, the model $L[M_1,\Sigma_{M_1}]$
given by fully closing $M_1$ under iteration strategy
constructs $M_1^\#$, and hence, in fact, $L[M_1,\Sigma_{M_1}]=L[M_1^\#]$. We now compute fairly sharp bounds on how much strategy
is needed to reach $M_1^\#$:

\begin{proof}[Proof of Theorem \ref{tm:nice_extensions}]\
	
	Part \ref{item:kappa^+_nice} ($M_1[\Sigma_{(\kappa^+)^{M_1}}]=M_1[\Sigma_{\kappa}]$):
		
			Let $\Tt\in M_1|(\kappa^+)^{M_1}$ be a maximal
	tree on $M_1$. Let $\gamma=\cof^{M_1}(\lh(\Tt))$.
	So $\gamma\leq\kappa$. In fact, $\gamma<\kappa$.
	For suppose $\gamma=\kappa$ and let $\delta=\delta(\Tt)$,
	so $\cof^{M_1}(\delta)=\kappa$.
	Let $j:M_1\to M_1$ be elementary with $\crit(j)=\kappa$
	and $j(\kappa_\alpha^{M_1})=\kappa_\alpha^{M_1}$ for $\alpha\geq\om$.
	Note then that $j(\Tt)\in M_1$ and $j(\Tt)$
	is a maximal tree on $M_1$. And $j$ is discontinuous
	at $\gamma$ and $\delta$. Let $\delta'=\sup 
	j``\delta$, so $\delta'<j(\delta)=\delta(j(\Tt))$.
	Let $N=L[M(\Tt)]$, so $j(N)=L[M(j(\Tt))]$.
	Let $s\in\vec{\mathscr{I}^{M_1}}$ with $\kappa_\om\leq\min(s)$.
	Then $\gamma^N_s<\delta$,
	and $j(s)=s$ and $j(\gamma^N_s)=\gamma^{j(N)}_s$.
	Therefore  letting $\Gamma=\mathscr{I}^{j(N)}\cut\kappa_\om$, we have
	\[ \delta'=\sup_{s\in\vec{\Gamma}}\gamma^{j(N)}_s=\delta^{j(N)}=j(\delta),\]
	a contradiction.
	
	Also, $\cof^{M_1}(\lh(\Tt))\geq\delta^{M_1}$.
	For suppose otherwise. Note that the existence
	of a maximal tree $\Uu$ on $M_1$ with $\cof(\lh(\Uu))<\delta^{M_1}$
	is a first order fact true in $M_1$,
	so there must be some such $\Uu\in M_1|\kappa$.
	But then it follows that $M_1[\Sigma_{\kappa}]$
	satisfies ``$\cof(\delta^{M_1})=\cof(\lh(\Uu))<\delta^{M_1}$, contradicting
	that this is a nice extension.
	
	So $\delta^{M_1}\leq\gamma=\cof^{M_1}(\lh(\Tt))<\kappa$.
	So working in $M_1$,
	we can fix some large $\eta\in\OR$
	and find an elementary $\pi:\bar{M}\to M_1|\eta$
	with $\Tt\in\rg(\pi)$ and $\rg(\pi)$ cofinal in $\lh(\Tt)$,
	and $\card(\bar{M})<\kappa$,
	and $\delta^{M_1}<\crit(\pi)$.
	Let $\pi(\bar{\Tt})=\Tt$.
	
	Now one can at this point argue that $\bar{\Tt}$ is maximal,
	by taking an appropriate ultrapower of any putative Q-structure
	by the extender derived from $\pi$. But actually it is easier
	to just deal with this  impossible case.
	So suppose $\bar{\Tt}$ is non-maximal. Then $b=\Sigma_{M_1}(\bar{\Tt})\in M_1$.
	But then note that $b'=\bigcup\pi``b\in M_1$ is a $\Tt$-cofinal
	branch, and since $\cof^{M_1}(\lh(\Tt))$ is uncountable,
	it follows that $M^\Tt_{b'}$ is wellfounded,
	so $b'=\Sigma_{M_1}(\Tt)$, which suffices.
	Now suppose instead that $\bar{\Tt}$ is maximal.
	Since $\bar{\Tt}\in M_1|\kappa$,
	we get $b=\Sigma_{M_1}(\bar{\Tt})\in M_1[\Sigma_{\kappa}]$.
	But by maximality, in $M_1[\Sigma_{\kappa}]$,
	$\lh(\bar{\Tt})$ has cofinality $\delta^{M_1}$,
	which is uncountable there, and therefore we can
	again define $b'=\bigcup\pi``b$, and get $b'=\Sigma_{M_1}(\Tt)$, as desired.

Part \ref{item:kappa^++1_not_nice}:
	We construct a  normal tree $\Tt$ on $M_1$,
	with  $\Tt$ definable
	over $M_1|(\kappa^+)^{M_1}$,
	$\lh(\Tt)=\delta(\Tt)=(\kappa^+)^{M_1}$,
	and such that $M_1^\#\in M_1[b]$, where $b=\Sigma_{M_1}(\Tt)$.

	The tree $\Tt$ is just that given by $M_1|\delta^{M_1}$-genericity
	iteration after linearly iterating the least measurable out to $\kappa$.
	For certainly $\Tt$ is maximal, with $\lh(\Tt)=\delta(\Tt)=\kappa^{+M_1}$.
	Let $P=L[M(\Tt)]$.
	
	We claim that $\mathscr{I}^P=\mathscr{I}^{M_1}\cut\{\kappa\}$.
	For
	\[ M_1=\Hull_{\Sigma_1}^{M_1}(\mathscr{I}^{M_1}), \]
	so using the forcing relation and since $\BB^P\sub\delta^P=(\kappa^+)^{M_1}$,
	\[ P=\Hull_{\Sigma_1}^P(\delta^P\cup\mathscr{I}^{M_1}\cut\{\kappa\}), \]
	but $\mathscr{I}^{M_1}\cut\{\kappa\}$ are
	$M_1$-indiscernibles with respect to the parameter $\kappa$,
	and therefore $\mathscr{I}^{M_1}\cut\{\kappa\}$
	are $P$-indiscernibles with respect to all parameters in $\delta^P$.
	So by  uniqueness, $\mathscr{I}^P=\mathscr{I}^{M_1}\cut\{\kappa\}$.
	But $\mathscr{I}^P=i_{M_1P}``\mathscr{I}^{M_1}$.
	Therefore
	\[ i_{M_1P}(\kappa_n^{M_1})=\kappa_{n+1}^{M_1} \]
	for all $n<\om$ (and $i_{M_1P}(\kappa_\alpha^{M_1})=\kappa_\alpha^{M_1}$
	for all $\alpha\geq\om$).
	
	Now in $M_1[b]$, we have $i_{M_1P}=i^\Tt_b$,
	and since $i^\Tt_b(\kappa_n^{M_1})=\kappa_{n+1}^{M_1}$,
	we can compute $\left<\kappa_n^{M_1}\right>_{n<\om}$,
	and therefore compute $M_1^\#$.
\end{proof}

\begin{rem} So let $\Tt,b$ be as just discussed,
and $\Uu$ the tree on $M_1$ iterating out to $M_\infty^{\mathscr{F}_{<\kappa}}$ and  $c=\Sigma_{M_1}(\Uu)$.
Then both $\Tt,\Uu$ are definable over $M_1|(\kappa^+)^{M_1}$,
but $M_1[c]$ is a nice extension, while $M_1[b]=L[M_1^\#]$.
\end{rem}

Instead of just adding a strategy fragment of the form $\Sigma_\eta$
to $M_1$, one can recursively close under strategy
 (giving the hierarchy of $L[M_1,\Sigma_{M_1}]$),
 and ask how large $\alpha$ needs to be to construct $M_1^\#$.
 Usually, one must also choose some appropriate manner of feeding in $\Sigma_{M_1}$, so that no novel information is accidentally encoded.
But we can ignore this here:
 Woodin showed that $L_\kappa[M_1,\Sigma_{M_1}]$ does not 
 construct $M_1^\#$, where $\kappa=\kappa_0^{M_1}$,
 but  we now show that actually, this is the same
 as $L_\kappa[M_1,\Sigma_\kappa]$, so this model is already closed:
 
\begin{tm}
Suppose that $M_1^\#$ exists and is $(\om,\om_1+1)$-iterable.
	Let $\kappa=\kappa_0^{M_1}$. Then
$L_\kappa[M_1,\Sigma_{M_1}]=L_\kappa[M_1,\Sigma_\kappa]$.
	\end{tm}
\begin{proof}
We will show that for each limit cardinal $\eta$ of $M_1$
	with $\delta^{M_1}<\eta<\kappa$,
	and each maximal $\Tt\in M_1[\Sigma_\eta]|\kappa$ on $M_1$,
	there is $\zeta<\kappa$
	such that $\Sigma_{M_1}(\Tt)\in M_1[\Sigma_\zeta]$. Clearly this suffices.

	Recall \emph{dl-stable} from \S\ref{sec:background} (adapted
	in the obvious way here).
	
	\begin{clmthree}
		There are cofinally many dl-stable $M_1$-limit cardinals $\eta<\kappa$.
	\end{clmthree}
	\begin{proof}
		Fix a limit cardinal $\eta\in(\delta^{M_1},\kappa)$ of $M_1$
		and $s\in\vec{\mathscr{I}^{M_1}}$
		such that $\eta=t^{M_1}(s)$.
		Given $P$ with $M_1\dashrightarrow P$ and $P|\delta^P\in M_1|\kappa$,
		define a sequence $\left<\eta^P_n\right>_{n<\om}$ as follows.
		Set $\eta^P_0=t^P(s)=i_{M_1P}(\eta)$, which is a limit cardinal
		of $P$ and $\delta^P<\eta_0^P<\kappa=\kappa_0^P$,
		and $\eta_0^{M_1}=\eta\leq\eta_0^P$.
		Note that for cofinally many $Q\in\mathscr{F}^{M_1}_{<\eta_0^{M_1}}$,
		$M_1$ is extender algebra generic over $Q$,
		and therefore $t^Q(s)$ is a limit cardinal of $M_1$.
		Working in $P$, given $\eta_n^P$, let
		\[ \eta_{n+1}^P=\sup_{Q\in\mathscr{F}^P_{<\eta_n^P}}t^Q(s). \]
		Since $P\in\mathscr{F}^P_{\eta_0^P}$,
		we get $\eta_n^P\leq\eta_{n+1}^P$, and note
		that if $\eta_n^P=\eta_{n+1}^P$ then $\eta_n^P=\eta_m^P$ for all $m\in[n,\om)$.
		Also let $\eta_\om^P=\sup_{n<\om}\eta_n^P$.
		Note that $\eta_\om^{M_1}$ is definable in $M_1$ from the parameter 
		$s$ (using the process we have just gone through).
		Note that each $\eta_n^{M_1}$ is a limit cardinal of $M_1$
		(using the cofinality of extender algebra grounds just mentioned above).
		
		We claim $\eta_\om^{M_1}$ is dl-stable. For given
		$P$ with $M_1\dashrightarrow P$ and $P|\delta^P\in M_1|\eta_\om^{M_1}$,
		for each $n\leq\om$,
		we have $\eta_n^P=i_{M_1P}(\eta_n^{M_1})$, since
		$i_{M_1P}(s)=s$ and by elementarity.
		Suppose  $\eta_\om^P>\eta_\om^{M_1}$.
		Then we may assume $M_1$ is extender algebra
		generic over $P$, and therefore each $\eta_n^P$ is also
		a limit cardinal of $M_1$. 
		Let $n<\om$ be least such that $\eta_n^P>\eta_\om^{M_1}$.
		If $n=0$, i.e. $t^P(s)>\eta_\om^{M_1}$, we contradict
		the definition of $\eta_1^{M_1}$.
		So we get $\eta_n^P\leq\eta_\om^{M_1}<\eta_{n+1}^P$.
		But then there is $Q\in P|\eta_n^P$ 
		with $\delta^Q<\eta_n^P$ such that $t^Q(s)>\eta_\om^{M_1}$,
		but note that then there is $m<\om$
		such that $Q|\delta^Q\in M_1|\eta_m^{M_1}$,
		so then $\eta_{m+1}^{M_1}\geq t^Q(s)>\eta_\om^{M_1}$,
		a contradiction.
	\end{proof}
	
	So we may assume that $\Tt\in M_1[\Sigma_\eta]|\kappa$
	where $\eta$ is dl-stable. Let $\mathscr{F}=\mathscr{F}^{M_1}_{<\eta}$
	and $M_\infty=M_\infty^{\mathscr{F}}$ and $\Sigma_\infty=\Sigma_\infty^{\mathscr{F}}$.
	By \S\ref{sec:background} (or its proof adapted here) and since $\eta$
	is dl-stable, the iteration map $j:M_1\to M_\infty$
	extends to $j^+:M_1[\Sigma_\eta]\to M_\infty[\Sigma_\infty]$,
	and $j^+(\kappa)=j(\kappa)=\kappa$.
	So
	\[ j^+(\Tt)\in M_\infty[\Sigma_\infty]|\kappa\sub M_1|\kappa.\]
	Therefore $c=\Sigma_{M_1}(j^+(\Tt))\in M_1[\Sigma_{\zeta}]$
	for some $\zeta<\kappa$. But then working in $M_1[\Sigma_\zeta]$,
	we can use $c$ to compute $b=\Sigma_{M_1}(\Tt)$, as desired.
\end{proof}

\section{Killing $(\om,\om_1+1)$-iterability with $\sigma$-distributive 
	forcing}\label{sec:kill_it}

A corollary of  \cite[Theorem 7.3]{iter_for_stacks}
is that if $M$ is an $\om$-sound premouse
which projects to $\om$ and $M$ is $(\om,\om_1+1)$-iterable
then $M$ is still $(\om,\om_1+1)$-iterable in any ccc forcing extension.
\footnote{We are working in $\ZFC$ in this paper, though
	\cite{iter_for_stacks} was more general than this.}

What if we replace the ccc with other properties?
It turns out that $\sigma$-distributive forcing
need not preserve the $(\om,\om_1+1)$-iterability
of $M_1^\#$. The following result is joint with Ralf Schindler:

\begin{tm}\label{tm:killing_it}
Suppose  $M_1^\#$ exists and is $(\om,\om_1+1)$-iterable.
Then there is an inner model $W$ such that $M_1^\#\in W$
and $W\sats\ZFC+$``$M_1^\#$ is $(\om,\om_1+1)$-iterable'',
but there is a forcing $\PP\in W$ such that $W\sats$``$\PP$ is $\sigma$-distributive and $\PP$ forces that $M_1^\#$ is not $(\om,\om_1+1)$-iterable''.
\end{tm}
\begin{proof}
We may assume  $V=L[A]$ for some $A\sub\om_1$.
For let $\Sigma=\Sigma_{M_1^\#}$.
Let $C$ code $\left<f_\alpha\right>_{\alpha<\om_1}$ with each
$f_\alpha:\om\to\alpha$
surjective. Let
\[ \Sigma'=\Sigma\inter L_{\om_1}[M_1^\#,\Sigma,C]\]
and $A=(M_1^\#,\Sigma',C)$.
So $\om_1^{L[A]}=\om_1$ and
$L[A]\sats\text{``}M_1^\#\text{ is }(\om,\om_1)\text{-iterable''}$.
But then
$L[A]\sats\text{``}M_1^\#\text{ is }(\om,\om_1+1)\text{-iterable''}$.
For if $\Tt$ on $M_1$ has length $\delta=\om_1$ then
$L[M(\Tt)]\sats\text{``}\delta\text{ is not Woodin''}$,
because otherwise  
$i^\Tt_b(\delta^{M_1})=\om_1$ where $b=\Sigma(\Tt)$,
but $i^\Tt_b(\delta^{M_1})$ has cofinality $\om$ in $V$.
But then using the Q-structure for $M(\Tt)$, we get the correct branch in $L[A]$.

For a non-dropping iterate $N$ of $M_1$ and $E\in\es^N$,
say $E$ is \emph{$A$-ml-bad} iff in $N$,
 $E$ induces a meas-lim extender algebra axiom $\varphi$ such that
$A\not\sats\varphi$ (in particular, $\nu_E$ is a limit of $N$-measurables).

Now work in $L[A]$.
We define a forcing $\PP$ which is $\sigma$-distributive and
kills the $(\om,\om_1+1)$-iterability of $M_1^\#$.
The conditions are iteration trees $\Tt$
such that:
\begin{enumerate}[label=--]
	\item $\Tt$ is on $M_1$, is normal, has countable successor length,
	$b^\Tt$ is non-dropping, and $\Tt$ is normally 
	extendible; that is,
	$\delta^{M^\Tt_\infty}>\sup_{\alpha+1<\lh(\Tt)}\lh(E^\Tt_\alpha)$.
	\item For every $\alpha+1<\lh(\Tt)$, $\es(M^\Tt_\alpha)$ has no $A$-ml-bad 
	extenders 
	indexed  $<\lh(E^\Tt_\alpha)$, and either
	\begin{enumerate}[label=(\roman*)]
		\item $E^\Tt_\alpha$ is $A$-ml-bad
	or \item $E^\Tt_\alpha$ is an $M^\Tt_\alpha$-total normal measure
	and $\crit(E^\Tt_\alpha)$ is not a limit of $M^\Tt_\alpha$-measurables.
	\end{enumerate}
\end{enumerate}
Set $\Tt\leq\Ss$ 
iff $\Tt$ extends $\Ss$.

Let $G$ be $L[A]$-generic for $\PP$. Let $\Tt_G=\bigcup G$.
Clearly $\Tt_G$ is a limit length normal tree on $M_1$.

\begin{clmthree}
	$\lh(\Tt_G)=\om_1$.
\end{clmthree}
\begin{proof}
	Fix $\Tt\in\PP$ and $\beta<\om_1$.
	Let $\Ss$ extend $\Tt$ via meas-lim extender algebra genericity iteration for 
	making $A$ generic,
	until reaching some $M^\Ss_\alpha$ with  a successor
	$M^\Ss_\alpha$-measurable
	$\kappa>\sup_{\beta<\alpha}\lh(E^\Ss_\beta)$
such that
	$M^\Ss_\alpha$ has no $A$-ml-bad extender with index ${<\kappa}$.
	Note that actually either $\alpha+1=\lh(\Tt)$ or $\alpha=\lh(\Tt)$.
	Set
	$\Ss=\Ss\rest(\alpha+1)\conc\Uu$,
	where $\Uu$ is the linear iteration of $M^\Ss_\alpha$
	of length $\beta+1$, at $\kappa$ and its images.
	Since $\kappa$ is not a limit of measurables,
	it easily follows that $\Ss\in\PP$ (the linear iteration does not leave any 
	$A$-ml-bad extenders behind),
	and we have $\Ss\leq\Tt$ and $\lh(\Ss)>\beta$, which suffices by density.
\end{proof}

Work again in $L[A]$.

\begin{clmthree}$\PP$ is $\sigma$-distributive.
\end{clmthree}
\begin{proof}
	Fix $\mathscr{D}=\left<D_n\right>_{n<\om}$
	with each $D_n\sub\PP$ dense.
	Let $\lambda\in\OR$ be large and
	\[ N=\cHull^{L_\lambda[A]}(\{\PP,\mathscr{D},A\}) \]
	and $\pi:N\to L_\lambda[A]$ be the uncollapse.
	So $\om_1^N=\crit(\pi)$. Let $\pi(\bar{A})=A$, etc. So 
	$\bar{A}=A\inter\om_1^N$ and $N=L_{\bar{\lambda}}[\bar{A}]$ for some 
	$\bar{\lambda}\in\OR$,
	and note that $N$ is pointwise definable from the parameter
	$(\bar{\PP},\bar{\mathscr{D}},\bar{A})$, and in particular countable.
	
	Let $g$ 
	be $N$-generic
	for $\bar{\PP}$. 
	Let $\Tt=\Tt_g$. By the preceding claim, $\Tt$ is a length 
	$\delta=\delta(\Tt)=\om_1^N$ tree on 
	$M_1$. Let $\eta$
	be such that either $\eta=\OR$
	or $Q=L_{\eta}[M(\Tt)]$ projects ${<\delta}$
	or is a Q-structure for $\delta$.
	Let $b=\Sigma(\Tt)$ (we have $b\in V=L[A]$). Note 
	$\Tt\conc b\in\PP$ iff
	$i^\Tt_b(\delta^{M_1})>\delta$
	iff $\eta<\OR$.
	Suppose $\eta=\OR$.
	Then $\delta$ is Woodin in $P=L[M(\Tt)]$
	and note that $\bar{A}$ is $P$-generic for $\BB_\delta^P$,
	so $\delta$ is regular in $P[\bar{A}]=L[M(\Tt),\bar{A}]$.
	But $N=L_{\bar{\lambda}}[\bar{A}]\in P[\bar{A}]$
	and $N$ is pointwise definable from
	parameters,
	so $\delta=\om_1^N$ is countable
	in $L[\bar{A}]\sub P[\bar{A}]$, a contradiction.
	So $\eta<\OR$ and $\Tt\conc 
	b\in\PP$.
	
	Since $g\inter D_n\neq\emptyset$,
	$\Tt\conc b$ extends some element of $D_n$, for each $n$, completing the 
	proof.
\end{proof}

Work in $L[A,G]$. Let $\Tt=\Tt_G$.
Let $\eta$ be least such that either $\eta=\OR$
or $Q=L_\eta[M(\Tt)]$ is a Q-structure for $M(\Tt)$.
\begin{clmthree}
	$\Tt$ is maximal. That is, $\eta=\OR$ and 
	$L[M(\Tt)]\sats$``$\delta(\Tt)$ is Woodin''.
\end{clmthree}
\begin{proof}
	Suppose not. Then by absoluteness there is a $\Tt$-cofinal branch $b$.
	Since $\om_1$ was preserved by $\PP$ (by 
	$\sigma$-distributivity),
	let
	$\pi:M\to L_\lambda[A,G]$
	be elementary with $M$ countable transitive, where $\lambda\in\OR$ is large
	and $A,\Tt,b\in\rg(\pi)$. Let $\kappa=\crit(\pi)$.
	By the usual argument, $i^\Tt_{\kappa\om_1}\rest\pow(\kappa)\sub\pi$,
	and clearly
	$\kappa=\delta(\Tt\rest\kappa)$.
	Let $\alpha+1=\successor^{\Tt}(\kappa,\om_1)$,
	so $\kappa=\crit(E^\Tt_\alpha)$.
	By the usual argument, $E^\Tt_\alpha$ was not chosen for genericity 
	iteration purposes,
	so by the construction of $\Tt$, $\kappa$ is a successor
	measurable of $M^\Tt_\alpha$.  By elementarity, it follows
	that $M(\Tt)$ has only boundedly many measurables.
	But since the $A$-ml-bad extenders $E$ have $\nu_E$
	a limit of measurables, this easily contradicts density.
\end{proof}

Now since $\Tt_G$ is maximal, there can be no $\Tt_G$-cofinal branch in $L[A,G]$
(as otherwise $i^{\Tt_G}_b$ would singularize $\om_1$).
Therefore $M_1$ and $M_1^\#$ are not $(\om,\om_1+1)$-iterable in $L[A,G]$, as 
desired.
\end{proof}

\begin{rem}The forcing $\PP$ above is not $\sigma$-closed. (For
by $(\om,\om_1+1)$-iterability in $L[A]$, $A$-ml-genericity iterations
terminate in countably many steps, and this easily yields descending $\om$-sequences
$\left<\Tt_n\right>_{n<\om}$ of $\PP$-conditions such that $\bigcup_{n<\om}\Tt_n$ is maximal, and hence with no extension in $\PP$.)
Does $\sigma$-closed forcing preserve the $(\om,\om_1+1)$-iterability
of $M_1^\#$?\end{rem}

\section{More iterates of $M_1$ around $L[x]$}

\begin{dfn}
	Let $x\in\RR$ with $x\geq_T M_1^\#$ and $\kappa\in\OR$. Say that $\kappa$ is \emph{Woodin-amenable} (for $L[x]$) iff $\kappa$
	is regular in $L[x]$ and there is  a maximal tree $\Tt$ on $M_1$
	with $\Tt\in L[x]$ and $\delta(\Tt)=\kappa$.
Say that $\kappa$ is \emph{distributively virtually Woodin-amenable} (for $L[x]$)
iff in $L[x]$, there is a ${<\kappa}$-distributive forcing $\PP$ which
forces the existence of such a tree.
\end{dfn}

\begin{rem}
	It is easy to see that we can add the requirement
	that $x$ be extender algebra generic over $L[M(\Tt)]$
	in both definitions above, without changing them
	(the $x$-genericity iteration performed with $L[M(\Tt)]$
	does not move the Woodin).

\end{rem}

\begin{tm}
	Let $x\in\RR$ with $x\geq_T M_1^\#$. Work in $L[x]$. Then:
	\begin{enumerate}
		\item\label{item:Woodin-amenable_through_theta_0} Every inaccessible  $\kappa\leq\theta_0$ is Woodin-amenable,
		where $\theta_0$ is the least Mahlo cardinal.
		\item\label{item:weak_compact_not_Woodin-amenable} No weakly compact cardinal is Woodin-amenable;
		in particular, $\kappa_0^{L[x]}$ is not, and hence measure one many
		$\kappa<\kappa_0^{L[x]}$ are not (with respect to the $x^\#$ measure).
		\item\label{item:stat-many_Woodin-amenable} There are stationarily many $\kappa<\kappa_0^{L[x]}$ which
		are Woodin-amenable.
	\end{enumerate}
	\end{tm}
\begin{proof}
	Part \ref{item:Woodin-amenable_through_theta_0} was proven in
	 Claim \ref{clm:iterate_height_delta} of the proof of Theorem \ref{tm:delta-cc_mantle}.
For part \ref{item:weak_compact_not_Woodin-amenable}, suppose $\kappa$
	is weakly compact and $\Tt$ a counterexample.
	Consider an embedding $\pi:M\to U$,
	with $M,U$ transitive and $\crit(\pi)=\kappa$ and $M(\Tt)\in M$,
	for a contradiction.
	
	Part \ref{item:stat-many_Woodin-amenable}: Let $C\sub\kappa_0^{L[x]}$
	be club. Then working much as in the proof of Claim \ref{clm:iterate_height_delta}
	of the proof of Theorem \ref{tm:delta-cc_mantle}, form an ml-genericity iteration
	which sends every measurable of $M(\Tt)$ to some element of $C$. 
	This has to terminate with a maximal tree $\Tt$ with $\delta(\Tt)<\kappa_0^{L[x]}$, by the previous part, and hence with $\delta(\Tt)\in C$.
	
	\end{proof}
\begin{tm}Let $x\geq_T M_1^\#$. Work in $L[x]$. Then every
	inaccessible is distributively virtually Woodin-amenable.
\end{tm}
\begin{proof}
	We define $\PP$ just like in the proof of Theorem \ref{tm:killing_it},
	except that we allow arbitrary trees of length ${<\kappa}$,
	and the ml-genericity iteration is with respect to $x$.
	
	Let's observe that $\PP$ is ${<\kappa}$-distributive.
	Let $\left<D_\alpha\right>_{\alpha<\gamma}$ be a sequence
	of dense sets, where $\gamma<\kappa$.
	Let $\Tt\in\PP$ be given. Now extend $\Tt$ with a descending
	sequence $\left<\Tt_\alpha\right>_{\alpha\leq\gamma}\sub\PP$ as follows.
	Let $\Tt_0\leq\Tt$ and $\delta(\Tt_0)>\gamma$.
Given $\Tt_\alpha$, let $\Tt_{\alpha+1}$ be  any $\Uu\in\PP$
	with $\Uu\leq\Tt_\alpha$ and $\Uu\neq\Tt_\alpha$ and $\Uu\in D_\alpha$.
Extend arbitrarily at limit stages. To see this works,
	let $\eta\leq\gamma$ be a limit ordinal and $\Uu=\bigcup_{\alpha<\eta}$;
	we just need to see that $\Uu$ is non-maximal, so suppose it is maximal.
	Then $x$ is generic over $L[M(\Uu)]$, and therefore $\delta(\Uu)$
	is regular in $L[M(\Uu)][x]=L[x]$. But 
	\[ \cof^{L[x]}(\delta(\Uu))=\cof^{L[x]}(\eta)\leq\gamma<\delta(\Tt_0)<\delta(\Uu),\]
	 so $\delta(\Uu)$ is singular in $L[x]$, contradiction.
	 
	 Since $\PP$ is ${<\kappa}$-distributive, $\kappa$ remains
	 inaccessible in $L[x,G]$, where $G$ is $(L[x],\PP)$-generic.
	 But then we can argue like in the proof of Theorem \ref{tm:killing_it}
	 to see that $\Tt_G$ is maximal.
	\end{proof}

\section{$\HOD^{L[x]}$ through the least strong of $M_\infty$}\label{sec:strong_cardinal}

It seems the theorem below is likely a consequence
of the proof of Woodin's \cite[Theorem 8.22]{lcfd},
but we give below a direct proof:

\begin{tm}
	Suppose $M_1^\#$ exists and is $(\om,\om_1+1)$-iterable.
	Assume constructible Turing determinacy.
	Then there is a cone of reals $x$ such that,
	letting $\mathscr{F}=\mathscr{F}_{\leq\omega_1^{L[x]}}$
	and $\kappa$ be the least ${<\delta_\infty}$-strong cardinal of $M_\infty=M_\infty^{\mathscr{F}}$, then
	 \[ V_{\kappa+1}^{\HOD^{L[x]}}=V_{\kappa+1}^{M_\infty}.\]
	\end{tm}

\begin{proof}
	\begin{clmfive}[Steel]\label{clm:OD_theory_Cohen_absolute}
		There is a cone of reals $x$ such that
		$L[x]\sats$``Cohen forcing does not change the theory of the ordinals'',
		and hence, ``Cohen forcing does not change $\HOD$''.\footnote{
		    Steel's orginal argument was sketched in \cite[Footnote 2, p.~603]{a_long_comparison}.
			The proof we give here was also sketched there, but
			we give a clearer proof here.}
		\end{clmfive}
	\begin{proof}
		Let $y\in\RR$. Let $N$ be an iterate of $M_1$ with $y\in N[g]$
		for some $g$ which is $(N,\PP)$-generic, where $\PP=\Coll(\om,\delta^N))$-generic.
		Let $x\in\RR$ with $L[x]=N[g]$. Let $h$ be $L[x]$-generic
		for Cohen forcing. Then there is $h'$ such that $g\cross h'$
		is $\PP\cross\PP$-generic
		over $N$ and $N[g,h']=L[h]$. But $\PP\cross\PP$ is equivalent 
		with $\PP$ and is homogeneous, so the ordinal theory of $L[x]$ is
		the same as that of $L[x,h]$, which by constructive Turing determinacy
		suffices.
		\end{proof}
	
Now fix $x$ in the cone of the claim and with $M_1^\#\in L[x]$.
We claim that $x$ works. For $M_\infty\sub\HOD^{L[x]}$,
so we just need to see that $V_{\kappa+1}^{\HOD^{L[x]}}\sub M_\infty$.
For this, it easily suffices to see that $\pow(\kappa)^{\HOD^{L[x]}}\sub M_\infty$. Suppose otherwise and let
 $X\sub\kappa$ and $\varphi$ be  a formula and $\eta\in\OR$ with
\[ \alpha\in X\iff L[x]\sats\varphi(\eta,\alpha), \]
but $X\notin M_\infty$. Since $M_\infty$ is $\OD$, by minimizing $\eta$,
we may assume $\eta=0$, so we drop that variable.

 Let $P\in\mathscr{F}$. 
Let $\kappa^P$ be the least ${<\delta^P}$-strong of $P$.
Let $\pi^P$ denote the map $P\to M_\infty^{\mathscr{E}^P}$,
where $\mathscr{E}^P$ denotes the system of iterates $Q$ of $P$
with $Q|\delta^Q\in P[g]$ and $\delta^Q=\om_1^{P[g]}$
for $g$ being $(P,\Coll(\om,\delta^P))$-generic.
Define
\[ X^P=\Big\{\alpha<\kappa^P\Bigm|P\sats\Coll(\om,\delta^P)\forces\varphi(\pi^P(\alpha))\Big\};\]
this is computed in $P$ (so $X^P\in P$),
as we only need $\pi^P\rest\kappa^P$, and $\pi^P\rest\kappa^P\in P$.
The following claim completes the proof:

\begin{clmfive}
	$i_{PM_\infty}(X^P)=X$.\end{clmfive}
\begin{proof}
	Let $\alpha<\kappa^P$. We want to see that $\alpha\in X^P$
	iff $i_{PM_\infty}(\alpha)\in X$. Now a first difficulty is that we might
	have $\delta^P=\om_1^{L[x]}$. But because $\alpha<\kappa^P$,
	we can get around this. Let $\Tt=\Tt_{M_1P}$, so $\Tt\in L[x]$ (and recall
	$M_1^\#\in L[x]$). Let $\gamma\in b^\Tt$ be such that $\kappa^P<\crit(i^\Tt_{\gamma\infty})$. Let $\bar{P}=M^\Tt_\gamma$.
	So $\kappa^{\bar{P}}=\kappa^P$ and $\bar{P}|\kappa^P=P|\kappa^P$,
	and $(\delta^{+\bar{P}})<\om_1^{L[x]}$. But  $i_{\bar{P}M_\infty}(\alpha)=i_{PM_\infty}(\alpha)$ because
	$\kappa^P$ is a limit of cutpoints of $P$ (here we use crucially
	that $\kappa^P$ is the least strong). Since $\delta^{+\bar{P}}$ is countable
	in $L[x]$, working in $L[x]$, we can form a Neeman genericity iteration,
	using only critical points $>\alpha$,
	producing an iterate $Q$ of $\bar{P}$. Working in $V$,
	let $g$ be $(Q,\Coll(\om,\delta^Q))$-generic with $x\in Q[g]$.
	
	Let $g_0,g_1$ be such that $g_0$ is $Q$-generic
	and $Q[g_0]=L[x]$ and $g_1$ is $Q[g_0]$-generic
	and $Q[g_0][g_1]=Q[g]=L[x][g_1]$. Then $L[x][g_1]$ is a Cohen
	forcing extension of $L[x]$ (or $g_1=\emptyset$). For letting $\BB\in Q$
	be such that $Q\sats$``$\BB$ is the regular open algebra
	associated to $\Coll(\om,\delta^Q)$'',
	then $\Coll(\om,\delta^Q)$ is dense in $\BB$,
	and hence the set of equivalence classes
	it induces is dense in $\BB/g_0$ (and $g_1$ is $(L[x],\BB/g_0)$-generic).
	But in $L[x]$, $\Coll(\om,\delta^Q)$ is countable,
	so in $L[x]$, $\BB/g_0$ has a countable dense subset,
	so is forcing equivalent to Cohen forcing there (or the trivial poset).
	
	So $Q[g]=L[x,h]$ for some Cohen generic $h$ over $L[x]$.
	Therefore by Claim \ref{clm:OD_theory_Cohen_absolute},
	the ordinal theory of $Q[g]$ is the same as that of $L[x]$.
	Note also that $\mathscr{F}$ is dense in $\mathscr{E}^Q$,
	and therefore \[ i_{PM_\infty}(\alpha)=i_{\bar{P}M_\infty}(\alpha)=i_{QM_\infty}(\alpha)=\pi^Q(\alpha),\]
	so $i_{PM_\infty}(\alpha)\in X$ iff $\pi^Q(\alpha)\in X$ iff $\alpha\in X^Q$
		iff $\alpha\in X^{\bar{P}}$ iff $\alpha\in X^P$, as desired.
	\end{proof}

 This completes the proof of the theorem.
	\end{proof}

Note that since $\kappa$ is measurable in $M_\infty\sub\HOD^{L[x]}$,
it follows that $\kappa$ is measurable in $\HOD^{L[x]}$.
Since $\kappa$ is a limit of cutpoints of $M_\infty$,
we have $\kappa\leq\kappa'$, where $\kappa'$
is the least ${<\delta}$-strong cardinal of $\HOD^{L[x]}$,
where $\delta=\om_2^{L[x]}$ (recall $\delta$ is Woodin in $\HOD^{L[x]}$).
And it follows that \emph{if} $\kappa=\kappa'$ then $V_\delta^{\HOD^{L[x]}}$
is the universe of a premouse.

\section{Comparison of countable iterates of $M_1$ in $L[x]$}

Let $x\in\RR$ with $M_1^\#\in L[x]$.
Let $P,Q$ be such that $M_1\dashrightarrow P,Q$ and $P,Q\in L[x]$
and $\delta^P,\delta^Q<\om_1^{L[x]}$.
Woodin asked whether
the pseudo-comparison of $P,Q$ always terminates
with length ${<\om_1^{L[x]}}$ (in connection
with the question of the nature of $\HOD^{L[x]}$).
We do not know the answer to this.
However, consider  the analogous
question, but where we have ``countably-in-$L[x]$ many''
iterates instead of ``two-many''. The answer here is ``no''.
In fact, there can be a sequence $\left<P_n\right>_{n<\om}\in L[x]$
such that $M_1\dashrightarrow P_n\dashrightarrow P_{n+1}$ 
and $\delta^{P_n}<\om_1^{L[x]}$
for each $n<\om$,
and letting $P_\om$ be the direct 
limit of the $P_n$'s (under the iteration maps),
then $P_\om\in L[x]$ and $\delta^{P_\om}=\om_1^{L[x]}$.
(And the result $P$ of their simultaneous comparison is just $P=P_\om$, by full normalization;
or just using commutativity, note that $P_\om$ embeds into $P$.)
To see this, just let $y\in\RR$
with $M_1^\#\in L[y]$,
let $G$ be $(L[y],\Coll(\om,\kappa))$-generic
where $\kappa$ is a limit of limit cardinals of $L[y]$,
let $x\in\RR$ with $L[x]=L[y,G]$, so $\om_1^{L[x]}=(\kappa^+)^{L[y]}$.
Working in $L[x]$, let $\left<\gamma_n\right>_{n<\om}$
be a strictly increasing sequence of limit cardinals of $L[y]$
with $\sup_{n<\om}\gamma_n=\kappa$,
and let $P_n$ be the direct limit of all 
$Q\in L_{\gamma_n}[y]$ with  $M_1\dashrightarrow Q$.
Then $P_\om$ is the direct limit of all
$Q\in L_{\kappa}[y]$ with $M_1\dashrightarrow Q$,
and therefore $\delta^Q=(\kappa^+)^{L[y]}=\om_1^{L[x]}$;
the sequence $\left<P_n\right>_{n\leq\om}\in L[x]$
because in $L[y]$, we can uniformly compute the direct
limit associated to a given limit cardinal $\gamma\leq\kappa$.

Also recall in this connection that by \cite{a_long_comparison},
there can be an $M_1$-like $N$ with $N|\delta^N\in\HC^{L[x]}$
and $M_1^\#\in L[x]$ and the pseudo-comparison of $M_1$ with $N$
lasts (exactly) $\om_1^{L[x]}$ stages.

\bibliographystyle{plain}
\bibliography{../bibliography/bibliography}

\end{document}